\documentclass[11pt,reqno]{amsart}
\usepackage{amsmath,amsthm,amssymb,mathrsfs,stmaryrd,color}
\usepackage[all]{xy}
\usepackage{url}
\usepackage{geometry}
\usepackage[utf8]{inputenc}
\usepackage[T1]{fontenc}

\usepackage{relsize}
\usepackage[bbgreekl]{mathbbol}
\usepackage{amsfonts}

\DeclareSymbolFontAlphabet{\mathbb}{AMSb}
\DeclareSymbolFontAlphabet{\mathbbl}{bbold}

\usepackage{enumitem}
\usepackage[colorlinks=true,hyperindex, linkcolor=magenta, pagebackref=false, citecolor=cyan,pdfpagelabels]{hyperref}
\usepackage[capitalize]{cleveref}
\usepackage{mathtools}
\usepackage{tikz-cd}
\setlist[enumerate]{itemsep=2pt,parsep=2pt,before={\parskip=2pt}}

\newcommand{\cosimp}[3]{\xymatrix@1{#1 \ar@<.4ex>[r] \ar@<-.4ex>[r] & {\ }#2 \ar@<0.8ex>[r] \ar[r] \ar@<-.8ex>[r] & {\ } #3 \ar@<1.2ex>[r] \ar@<.4ex>[r] \ar@<-.4ex>[r] \ar@<-1.2ex>[r] & \cdots }}

\newcommand{\adjunction}[4]{\xymatrix@1{#1{\ } \ar@<0.3ex>[r]^{ {\scriptstyle #2}} & {\ } #3 \ar@<0.3ex>[l]^{ {\scriptstyle #4}}}}

\begin{document}

\newtheorem{theorem}{Theorem}[section]
\newtheorem*{theorem*}{Theorem}
\newtheorem*{definition*}{Definition}
\newtheorem{proposition}[theorem]{Proposition}
\newtheorem{lemma}[theorem]{Lemma}
\newtheorem{corollary}[theorem]{Corollary}

\theoremstyle{definition}
\newtheorem{definition}[theorem]{Definition}
\newtheorem{question}[theorem]{Question}
\newtheorem{remark}[theorem]{Remark}
\newtheorem{warning}[theorem]{Warning}
\newtheorem{example}[theorem]{Example}
\newtheorem{notation}[theorem]{Notation}
\newtheorem{convention}[theorem]{Convention}
\newtheorem{construction}[theorem]{Construction}
\newtheorem{claim}[theorem]{Claim}
\newtheorem{assumption}[theorem]{Assumption}


\def\todo#1{\textcolor{red}%
{\footnotesize\newline{\color{red}\fbox{\parbox{\textwidth}{\textbf{todo: } #1}}}\newline}}

\def\commentbox#1{\textcolor{red}%
{\footnotesize\newline{\color{red}\fbox{\parbox{\textwidth}{\textbf{comment: } #1}}}\newline}}

\newcommand{\qc}{q-\mathrm{crys}}

\newcommand{\Shv}{\mathrm{Shv}}
\newcommand{\Vect}{\mathrm{Vect}}
\newcommand{\Higgs}{\mathrm{Higgs}}
\newcommand{\et}{\mathrm{\acute{e}t}}
\newcommand{\eh}{\mathrm{\acute{e}h}}
\newcommand{\proet}{\mathrm{pro\acute{e}t}}
\newcommand{\fppf}{\mathrm{fppf}}
\newcommand{\crys}{\mathrm{crys}}
\renewcommand{\inf}{\mathrm{inf}}
\newcommand{\Hom}{\mathrm{Hom}}
\newcommand{\Sch}{\mathrm{Sch}}
\newcommand{\fSch}{\mathrm{fSch}}
\newcommand{\Rig}{\mathrm{Rig}}
\newcommand{\Spf}{\mathrm{Spf}}
\newcommand{\Spa}{\mathrm{Spa}}
\newcommand{\Spec}{\mathrm{Spec}}
\newcommand{\Bl}{\mathrm{Bl}}
\newcommand{\sn}{{\mathrm{sn}}}
\newcommand{\perf}{\mathrm{perf}}
\newcommand{\Perf}{\mathrm{Perf}}
\newcommand{\Pic}{\mathrm{Pic}}
\newcommand{\qsyn}{\mathrm{qsyn}}
\newcommand{\perfd}{\mathrm{perfd}}
\newcommand{\arc}{{\rm arc}}
\newcommand{\conj}{\mathrm{conj}}
\newcommand{\rad}{\mathrm{rad}}
\newcommand{\Id}{\mathrm{Id}}
\newcommand{\coker}{\mathrm{coker}}
\newcommand{\im}{\mathrm{im}}
\newcommand{\Cond}{\mathrm{Cond}}
\newcommand{\CHaus}{\mathrm{CHaus}}
\newcommand{\cg}{\mathrm{cg}}
\newcommand{\topo}{\mathrm{top}}
\newcommand{\Top}{\mathrm{Top}}
\newcommand{\Ab}{\mathrm{Ab}}
\newcommand{\Set}{\mathrm{Set}}
\newcommand{\Fl}{\mathrm{Fl}}

\newcommand{\Pro}{\mathrm{Pro}}
\newcommand{\Ind}{\mathrm{Ind}}
\newcommand{\Sm}{\mathrm{SmRig}}
\newcommand{\dR}{{\mathrm{dR}}}
\newcommand{\HTlog}{\operatorname{HTlog}}
\newcommand{\HT}{\operatorname{HT}}
\newcommand{\bdr}{{\mathbb{B}_{\mathrm{dR}}^{+,\bullet}}}
\newcommand{\an}{{\mathrm{an}}}
\newcommand{\uHom}{\underline{\mathrm{Hom}}}
\newcommand{\Sym}{\operatorname{Sym}}
\newcommand{\Lie}{\operatorname{Lie}}
\setcounter{tocdepth}{1}

\newcommand{\calO}{\mathcal{O}}

\title{A $p$-adic Simpson correspondence for singular rigid-analytic varieties}
\author{Hanlin Cai}
\address{Department of Mathematics, Columbia University, New York, NY 10027}
\email{hc3589@columbia.edu}
\author{Zeyu Liu}
\address{Department of Mathematics, UC Berkeley, Berkeley, CA 94720}
\email{zeliu@berkeley.edu}
\begin{abstract}
 Let $C$ be a complete, algebraically closed non-archimedean extension of $\mathbb{Q}_p$, and $X$ be a proper rigid-analytic variety over $C$. We show that  the category of pro-\'etale vector bundles on $X$ is equivalent to the category of Higgs bundles on the $\eh$-site of $X$, thereby generalizing the work of Faltings and Heuer to arbitrary proper rigid-analytic varieties. 
\end{abstract}

\maketitle

\tableofcontents

\section{Introduction}
Let $C$ be a complete, algebraically closed non-archimedean extension of $\mathbb{Q}_p$, and $X$ be a rigid-analytic variety over $C$, namely an adic space locally of topologically of finite type over $\Spa(C,\mathcal{O}_C)$. Faltings studied the $p$-adic Simpson correspondence for smooth projective curves in \cite{faltings2005p}, which states that certain "generalized representations" correspond to Higgs bundles. Later, Abbes, Gros and Tsuji generalized to higher-dimensional algebraic varieties in \cite{abbes2016p}. Finally Heuer generalized to arbitrary smooth proper rigid-analytic varieties over $C$ in \cite{heuer2023p}. The goal of this article is to generalize the $p$-adic Simpson correspondence to the case where $X$ is not necessarily smooth. The underlying idea goes back to Deligne, who studied Hodge theory using the $\eh$-topology and hypercovers.

\begin{theorem}[{\cref{thm. p-adic simpson for proper rigid-analytic variety}}]
\label{introthm.non-smooth p-adic simpson correspondence}
     Let $X$ be a proper rigid-analytic variety over $C$. Fix a $1$-truncated smooth proper $\eh$-hypercover of $X$ with a lift to $B^+_\dR/\xi^2$ and an exponential map $\exp:C\to 1+\mathfrak{m}$. Then there exists an equivalence of symmetric monoidal categories:
    \begin{eqnarray*}
\Vect(X_\proet) \xrightarrow[]{\simeq } \Higgs(X_\eh).
	\end{eqnarray*}

\end{theorem}

The left-hand side is not new. The primitive comparison theorem of Scholze shows that the category of $C$-local systems is contained in the category of pro-\'etale vector bundles. And it is studied by Faltings under the name of "generalized representations". As for the right-hand side, it can be viewed as a Higgs field in which the sheaf of differentials is replaced by the first Deligne–Du Bois complex. Recall the $\eh$-topology studied previously by \cite{geisser2006arithmetic} and \cite{guo2019hodge} on the category $\Rig$ of rigid analytic varieties consists of \'etale coverings, blowup coverings and universal homeomorphisms. Hence by the resolution of singularities, the $\eh$-site on the category of rigid analytic varieties is locally smooth, i.e. it admits a basis of smooth rigid-analytic varieties. Therefore to define a (pre)sheaf on $\Rig_\eh$, it suffices to define it one a basis. Hence following \cite{guo2019hodge} we will denote the absolute differential sheaf $\Omega^i_\eh$ as the sheaf associated to the presheaf 
$$\Sm \to \Ab $$
sending $Y\in \Sm$ to $\Gamma (Y,\Omega^j_Y)$. In particular, after localizing at some rigid-analytic varieties $X$, one obtains an abelian sheaf on $X_\eh$. Since the $\eh$-topology is obviously finer than the \'etale topology on $X$, there is a natural projection $\rho:X_\eh \to X_\et$. It was proved by \cite{guo2019hodge} that if $X$ is the analytification of a projective variety over $C$ and there is an abstract isomorphism between $C$ and the complex numbers $\mathbb{C}$, then the derived pushforward $R\rho_*\Omega^i_\eh$ can be identified with the $i$-th Deligne–Du Bois complex under GAGA. This justifies the intuition behind the $\eh$-differential sheaf. And we can now define the Higgs bundles on $X_\eh$ as the \'etale case.
\begin{definition}
    A Higgs bundle $(E,\theta)$ on $X_\eh$ is a finite locally free $\mathcal{O}_{X_\eh}$-module $E$ equipped with a $\mathcal{O}_{X_\eh}$-linear morphism 
    $$\theta:E\to E\otimes\Omega^1_{X_\eh}(-1)$$
    such that $\theta\wedge\theta=0$. And we denote the groupoid of Higgs bundles on $X_\eh$ as $\Higgs(X_\eh)$.
\end{definition}

Note that when $X$ is smooth, the category of Higgs bundles on $X_\eh$ is equivalent to the category of Higgs bundles on $X_{\et}$ (\cref{prop. compare eh higgs with et higgs}). Thus, \cref{introthm.non-smooth p-adic simpson correspondence} generalizes the $p$-adic Simpson correspondence by \cite{heuer2023p} to the case of possibly singular proper rigid-analytic varieties.

\begin{remark}
We remark that even for algebraic varieties over the complex numbers, very little is known about Simpson’s correspondence in the singular case. For example, in \cite{greb2019nonabelian} such a correspondence is studied for projective varieties with klt singularities, where the sheaf of differentials is replaced by its reflexification. However, by \cite[Theorem 5.4]{HuberJorderDifferentialFormsHTopology}, the $h$-differential sheaf (which is the same as the $\eh$-differential sheaf), when restricted to a variety with klt singularities, coincides with the reflexive differentials. Thus, it seems plausible that applying our methods to algebraic varieties over the complex numbers would yield a generalization of the results of \cite{greb2019nonabelian}. Moreover, $h$-differential forms also appear in the study of (higher) $F$-injective singularities in positive characteristic in the work of \cite{kawakami2024higher}, so it is reasonable to expect a version of singular non-abelian Hodge theory in positive characteristic generalizing the work of \cite{ogus2007nonabelian}.
\end{remark}

\begin{remark}

In forthcoming work of Bhargav Bhatt and Mingjia Zhang, they construct a $\mathbb{G}_m$-gerbe $\mathcal{P}$ over the cotangent bundle of a smooth rigid-analytic variety $X$ over $C$, which they call the Simpson gerbe. This gerbe gives an equivalence between perfect complexes on the pro-\'etale site and perfect complexes of $\mathcal{P}$-twisted Higgs bundles. It is also reasonable to expect that an $\eh$-sheafified version of this gerbe would yield a corresponding generalization to arbitrary rigid-analytic varieties.
\end{remark}

As a $p$-adic analogue of the Hodge decomposition from complex geometry, Scholze \cite{scholze2013adic} proved that choosing a $B^+_\dR/\xi^2$ lift $\mathbb{X}$ of a smooth proper rigid-analytic variety $X$ over $C$ induces an isomorphism
\begin{equation*}
    H_{\et}^n(X, \mathbb{Q}_p) \otimes_{\mathbb{Q}_p} C=\bigoplus_{i+j=n} H^i(X, \Omega_X^j(-j)),
\end{equation*}
which was later generalized to the case that $X$ is not necessarily smooth by \cite[Theorem 1.1.3]{guo2019hodge}. From this perspective, Heuer's $p$-adic Simpson correspondence can be reinterpreted as a $p$-adic non-abelian Hodge correspondence, giving a non-abelian categorical generalization of Scholze's result above. Similarly, our \cref{introthm.non-smooth p-adic simpson correspondence} serves as a non-abelian generalization of \cite[Theorem 1.1.3]{guo2019hodge}, first justified by the following result:
\begin{theorem}[{\cref{thm. cohomology comparison}}]\label{introthm. cohomology comparison}
    Let $X$ be a proper rigid-analytic variety over $C$ of dimension $n$. Take a $3n+2$-truncated smooth proper $\eh$-hypercover of $X$ with a lift to $B^+_\dR/\xi^2$. Given a pro-\'etale vector bundle $\mathcal{E}$ on $X$, it corresponds to a Higgs bundle $(E,\theta)$ on $X_\eh$ under the equivalence in \cref{introthm.non-smooth p-adic simpson correspondence}, then there is a natural isomorphism
    \begin{equation*}
        R\Gamma_\proet(X, \mathcal{E})\xrightarrow{\simeq} R\Gamma_\eh(X, (E,\theta)).
    \end{equation*}
\end{theorem}

Like in the algebraic case, the key idea to prove the results above is to use local smoothness of the $\eh$-topology. Namely, given a proper rigid-analytic variety $X$ over $C$, one can construct a truncated smooth $\eh$-hypercover. And on smooth rigid-analytic varieties, one then argues that the category of vector bundles (resp. Higgs bundles) on $X_\et$ is equivalent to the category of vector bundles (resp. Higgs bundles) on $X_\eh$ as in \cref{thm.descent of vector bundles} and \cref{prop. compare eh higgs with et higgs}. Finally using the functoriality of the $p$-adic Simpson correspondence of Heuer, one can finally produce a functor between pro-\'etale vector bundles on $X$ and Higgs on bundles on $X_\eh$. And using descent results again we conclude the desired categorical equivalence.
\begin{remark}
    We warn the reader here that there are some subtleties that arise when dealing with general vector bundles on $X$ rather than the structure sheaf, even when $X$ is smooth. This is caused by the fact that $\Omega^1_\eh$ on $X$ is not a vector bundle even in the smooth case. Note that the rank may vary when restricted to different objects in $X_\eh$. See Warning \cref{warning. big differential and relative differential} for more details. Hence as opposed to the \'etale Dolbeault complex which is bounded, the $\eh$-Dolbeault complex 
    $$ E\to E\otimes \Omega^1_{X_\eh}(-1)\to E\otimes \Omega^2_{X_\eh}(-2)\to \cdots$$
    of some $\eh$-Higgs bundle $(E,\theta)$ is not necessarily bounded. Therefore, to compute the Dolbeault cohomology, we must first compare it with the pro-\'etale cohomology of the corresponding pro-\'etale vector bundle and rely on the cohomological bounds of the latter, c.f. \cref{prop.cohomological dimension bounds of proet VB}.
\end{remark}

\subsection{Organization of the article}
In Section 2, we recall the $\eh$-topology and establish the crucial descent properties for both vector bundles and Higgs bundles on smooth rigid-analytic varieties. In Section 3, we develop the necessary machinery for simplicial rigid-analytic varieties and establish the $p$-adic Simpson correspondence in the smooth proper simplicial setting. Finally, in Section 4, we synthesize these ingredients by utilizing smooth proper $\eh$-hypercovers to prove our main result---the extension of the $p$-adic Simpson correspondence to arbitrary proper rigid-analytic varieties---and conclude by establishing the corresponding cohomological comparisons and dimension bounds.

\begin{convention}
    Throughout this article, $C$ is a complete algebraically closed non-archimedean extension over $\mathbb{Q}_p$. Denote $\Spa(C)$ as the adic spectrum of $(C,\mathcal{O}_C)$. We let $\Rig$ denote the rigid-analytic varieties over $\Spa(C)$, namely adic spaces locally of finite type over $\Spa(C)$. We denote smooth rigid-analytic varieties as $\Sm$. Since rigid-analytic varieties are locally of topological finite type over $\Spa(C)$, any integrally closed open subring $A^+$ of $A$ that is topologically of finite type over $\mathcal{O}_C$ is $A^\circ$. Hence if not specified, we write $\Spa(A)$ to denote $\Spa(A,A^\circ)$. In the following we will write tft as topological of finite type. Moreover every morphism between tft is tft by \cite[Lemma 3.5 (iv)]{huber1994generalization}. If $X=\Spa(A)$ is an affinoid rigid-analytic variety, then we will denote $\Spec(A)$ as $X_a$. Moreover, if $Y$ is a scheme locally of finite type over $\Spec(A)$, we will follow \cite[Proposition 3.8]{huber1994generalization} to denote the relative analytification of $Y$ over $X$ as $Y^{\mathrm{an}/X}$.

    Vector bundles on a ringed topos $(\mathcal{T},\mathcal{O})$ always mean finite locally free $\mathcal{O}$-modules. On pro-\'etale site we only consider the completed structure sheaf $\hat{\mathcal{O}}_X$. For any rigid-analytic variety $X$, we will follow \cite{scholze2017etale} to denote $X^\diamond$ as its associated diamond. For any coherent sheaf $F$ on some rigid-analytic variety $X$, its cohomology on some rigid-analytic variety $f:Y\to X$ over $X$ is given by $f^*F$. Projections from the big $\tau$-site of $X$ to the small $\tau$-site of $X$ are denoted as $\Id_X$ when there is no confusion on the topology. 
    
    We will denote the 2-category (1-)groupoids as $\mathrm{Groupoids}$. Limits of groupoids are interpreted in $\mathrm{Groupoids}$.
\end{convention}
\subsection{Acknowledgement}
We would like to thank Brian Conrad, Johan de Jong, Ben Heuer, Sean Howe and Bogdan Zavyalov for helpful discussions and valuable feedback on a preliminary version. Special thanks to Bogdan Zavyalov for helpful discussions and Brian Conrad for suggesting useful references. Both of the authors are partially funded by the Simons Collaboration on Perfection during the preparation of the project. The second-named author was partially supported by NSF DMS-2053473 under Kiran Kedlaya.

\newpage

\section{$\eh$-descent of vector bundles and Higgs bundles}
\subsection{Recollection of $\eh$-topology}
In this section, we briefly recall the $\eh$-topology previously studied for schemes by \cite{geisser2006arithmetic} and for rigid-analytic varieties by \cite{guo2019hodge}. Recall that for $X$ a rigid-analytic variety and $\bigoplus_{d\ge 0}A_d$ be a graded coherent $\mathcal{O}_Y$-algebra, we denote 
    $$\underline{\mathrm{Proj}}^\mathrm{an}_Y(\bigoplus_{d\ge 0}A_d)\to X$$
    to be the relative projective space of $\bigoplus_{d\ge 0}A_d$ over $X$. If $Y$ is affinoid, then it is given by 
    $$\underline{\mathrm{Proj}}^\mathrm{an}_Y(\bigoplus_{d\ge 0}A_d)=(\underline{\mathrm{Proj}}_{\Spec(\mathcal{O}_Y(Y))}(\bigoplus_{d\ge 0}A_d(Y)))^{\mathrm{an}/Y}.$$
    We will write $\mathrm{Bl}_Z(X)$ as the blowup of $X$ along $Z$ where $Z$ is a Zariski closed subset of $X$ defined by a coherent ideal sheaf $\mathcal{I}$, namely
    $$ p_Z:\mathrm{Bl}_Z(X)\coloneqq\underline{\mathrm{Proj}}^\mathrm{an}_X(\bigoplus_{n\ge 0}\mathcal{I}^n) \to X.$$
And we will call the blowup $\Bl_Z(X)$ a smooth blowup if the center $Z$ is smooth.
\begin{definition}
    The $\eh$-topology on the category $\Rig$ of rigid-analytic varieties over $C$ is the Grothendieck topology associated to the pretopology generated by the following coverings:
    \begin{enumerate}
        \item \'Etale coverings.
        \item Blowup coverings, namely the coverings $\{Z\to X,\ \Bl_Z (X)\to X \}$ associated to the blowup square
        \[
        \begin{tikzcd}
            E \ar[d]\ar[r] & \Bl_Z (X) \ar[d]\\
            Z \ar[r] & X.
        \end{tikzcd}
        \]
        where $Z$ is the blowup center and $E$ is the exceptional divisor.
        \item Universal homeomorphisms.
    \end{enumerate}
    The big $\eh$-site on $\Rig$ is denoted as $\Rig_\eh$ and its localization at some rigid-analytic variety $X$ is denoted as $X_\eh\coloneq \Rig_\eh/X$.
\end{definition}

We will use the structure theorem of $\eh$-topology repeatedly. 
\begin{proposition}[{\cite[Theorem 2.4.11, Corollary 2.4.12]{guo2019hodge}}]
\label{prop. guo decompose}
    Let $X$ be a quasicompact rigid-analytic variety. Then any $\eh$-covering can be refined by 
    $$Y\to X_\mathrm{red}\to X$$
    where $Y$ is smooth and $Y\to X_\mathrm{red}$ a composition finitely many \'etale coverings and coverings associated to smooth blowups. In particular, if $X$ is smooth then any $\eh$-covering can be refined by a composition of finitely many \'etale coverings and coverings associated to smooth blowups.
\end{proposition}

The following corollary states that $\eh$-sheaf has blowup excision property.
\begin{corollary}
\label{prop.h-sheaf=fppf sheaf + abstract blowups}

    For any sheaf $F$ of spaces on $\Rig_\eh$, it satisfies blowup excisions, namely for an affinoid rigid-analytic variety $X$ and any blowup $p_Z:\Bl_Z(X)\to X$ along a nowhere dense Zariski closed subset $Z\hookrightarrow X$ with exceptional divisor $E\hookrightarrow \Bl_Z(X)$, the following diagram
\[
\begin{tikzcd}
    F(X)\arrow[d]\arrow[r]& F(Z)\arrow[d]\\ F(\Bl_Z(X))\arrow[r]& F(E)
\end{tikzcd}
\]
is a pullback square.

\end{corollary}
\begin{proof}
Let ${\underline{X}}$ be the $\eh$-sheafification of the representable presheaf (of sets) of $X$. Since ${\underline{X}}$ is the final object in the category of sheaves of spaces on $X_\eh$, we have $\mathrm{Map}({\underline{X}},F)=F(X)$. Therefore, it suffices to show that ${\underline{X}}$ is the homotopy pushout of ${\underline{Z}}$ and ${\underline{\Bl_Z(X)}}$ over ${\underline{E}}$ in the category of sheaf of spaces on $\Rig_\eh$. However, for any $Y\in \Rig$, we know that ${\underline{E}}(Y)\to {\underline{\Bl_Z(X)}}(Y)$ is injective and hence cofibrant since $E \hookrightarrow\Bl_Z(X)$ is a closed immersion. Hence the homotopy pushout of ${\underline{Z}}(Y)$ and ${\underline{\Bl_Z(X)}}(Y)$ over ${\underline{E}}(Y)$ coincides with its set-theoretic pushout. Thus the sheaf of the homotopy pushout of ${\underline{Z}}$ and ${\underline{\Bl_Z(X)}}$ over ${\underline{E}}$ is discrete and therefore coincides with its pushout in the category of sheaves of sets.

Finally, we verify that $\underline{X}$ is the pushout of $\underline{Z}$ and ${\underline{\Bl_Z(X)}}$ over ${\underline{E}}$ as sheaf of sets. Take $F$ be a sheaf of sets on $\Rig_\eh$, then the sheaf condition shows that $F(X)$ is the equalizer of 
$$F(Z)\times F(\Bl_Z(X)) \substack{\longrightarrow \\[-1em] \longrightarrow} F(Z)\times F(E)\times F(E)\times F(\Bl_Z(X)\times_X \Bl_Z(X)).$$
Specifically speaking, sections in $F(X)$ are equivalent of sections $(s_Z,s_B)\in F(Z)\times F(\Bl_Z(X))$ such that $s_Z|_E=s_B|_E$ and $\mathrm{pr}^*_1(s_B)=\mathrm{pr}^*_2(s_B)$ where $\mathrm{pr}_i$ are the projections from $\Bl_Z(X)\times_X \Bl_Z(X) \to \Bl_Z(X)$. Hence we need to verify that the second condition is automatically checked. We first claim that 
$$\Bl_Z(X)\coprod E\times_Z E \xrightarrow{\Delta\coprod i} \Bl_Z(X)\times_X \Bl_Z(X)$$
is an $\eh$-covering where $\Delta$ is diagonal and $i$ is the closed immersion. First observe that since the restriction of $\Bl_Z(X)$ over $X-Z$ is isomorphic go $X-Z$, the restriction of the fibre product $\Bl_Z(X)\times_X \Bl_Z(X)$ over $X-Z$ is again $X-Z$ which is covered by the diagonal $\Delta(\Bl_Z(X))$. Thus $\Delta\coprod i$ is surjective. Since $\Delta(E)$ is the intersection of the two Zariski closed subspace $\Delta(\Bl_Z(X)$ and $E\times_Z E$, the blowup of $\Delta(E)$ on $\Bl_Z(X)\times_X \Bl_Z(X)$ which can be identified with $\Bl_Z(X)\coprod \Bl_{\Delta(E)}(E\times_Z E)$. Therefore, the covering associated to $\Delta\coprod i$ can be refined by a blowup covering, i.e. it is an $\eh$-covering as claimed. Hence to show that two sections $\mathrm{pr}^*_1(s_B) $ and $\mathrm{pr}^*_2(s_B)$ are the same in $F(\Bl_Z(X)\times_X \Bl_Z(X))$, it suffices to show that they are the same after pulling back to $$F(\Bl_Z(X)\coprod E\times_Z E)=F(\Bl_Z(X))\times F( E\times_Z E).$$
However, it is clear that $\Delta^*\mathrm{pr}^*_1(s_B)=\Delta^*\mathrm{pr}^*_2(s_B)$ as $\Delta$ is the diagonal. To check that $i^*\mathrm{pr}^*_1(s_B) =i^*\mathrm{pr}^*_2(s_B)$, it suffices to check that after pulling back $s_B|_E$ to $E\times_Z E$ via two projections give the same section. However, this follows from the fact that $s_B|_E=s_Z|_E$.

\end{proof}

\begin{remark}
\label{rem.eh vs v topology}
    We remark that the $\eh$-topology is coarser than the $v$-topology of \cite{scholze2017etale}. If $p$ is an $\eh$-cover then $p^\diamond: Y^\diamond\to X^\diamond$ is a $v$-cover of diamonds, i.e., a surjective map of $v$-stacks. Indeed, by \cite[Lemma 12.11.]{scholze2017etale}, it suffices to show that $p^\diamond$ is quasicompact. For any quasicompact $v$-stacks $Z$ with a morphism $f:Z\to X^\diamond$, we need to prove that $Z\times_{X^\diamond}Y^\diamond$ is quasicompact. By covering $X^\diamond$ with affinoid open sets, we may assume that $Z$ maps into an affinoid open set $U^\diamond$ of $X^\diamond$. Then $Z\times_{X^\diamond}Y^\diamond\simeq Z\times_{U^\diamond}(p^{-1}(U))^\diamond$ is quasicompact since $U^\diamond$ is quasiseparated by \cite[Proposition 15.4.]{scholze2017etale}.
\end{remark}

\subsection{Descent of vector bundles}
To justify our generalization to the singular case, we verify that Higgs bundles in the smooth case satisfy $\eh$-descent. 
We first verify that vector bundles satisfy $\eh$-descent, i.e. the groupoid of étale vector bundles $\Vect(X_{\et})$ is an $\eh$-stack. To this end, we follow the strategy in \cite{bhatt2017projectivity}. We will denote the natural projection as $\rho_X:X_\eh\to (\Rig/X)_\et$.



\begin{theorem}
\label{thm.descent of vector bundles}
    For any smooth rigid-analytic variety $X$, there is an equivalence of groupoids
    $$\rho_{X}^*:\Vect((\Rig/X)_\et)\xrightarrow{\simeq} \Vect(X_\eh) .$$
    In fact, for any \'etale vector bundle $E$, we have $R\rho_{X,*}\rho_{X}^* F= F$. 
\end{theorem}

\begin{corollary}
    The presheaf 
    $$\Vect: \Sm_\eh \to \mathrm{Groupoids}$$
    sending $Y$ to $\Vect(Y)$ is an $\eh$-sheaf.
\end{corollary}

\begin{remark}
\label{rem.unfolding hypersheaf}
    Since $\Sm$ does not admit fibre products, in order to talk about the $\eh$-topology on it, one can either use the language of covering sieves or require it to be an $\eh$-hypersheaf. These notions are equivalent by \cite[Lemma 6.5.2.9]{lurie2006higher} since the value is $2$-truncated. Since $\Sm$ forms a basis in $\Rig_\eh$ by the local smoothness of the $\eh$-topology, one can unfold a (hyper)sheaf on $\Sm_\eh$ to a (hyper)sheaf on $\Rig_\eh$ which induces an equivalence
    $$\Shv(\Sm_\eh,\mathrm{Groupoids})\xrightarrow{}\Shv(\Rig_\eh,\mathrm{Groupoids}).$$
\end{remark}
\begin{remark}
    One can give an alternative proof of \cref{thm.descent of vector bundles} by arguing that vector bundles on smooth rigid-analytic varieties satisfy $v$-descent. Indeed, vector bundles of smooth rigid-analytic varieties are equivalent to pro-\'etale vector bundles having a trivial geometric Sen operator by \cite[Theorem 4.8]{heuer2023p}. However, the latter prestack satisfies $v$-descent by arguing similarly as in \cite[Theorem 3.5.8]{kedlaya2016relative2}. Such facts are also observed in \cite[Theorem 7.13]{heuer2022moduli}, where we refer the reader for more details. The proof we give in the following is more geometric and intrinsic in flavor.
\end{remark}

By the following lemma, we will not distinguish vector bundles on (big) \'etale and (big) analytic sites of $X$.

\begin{lemma}[{\cite[Proposition 8.2.3]{fresnel2003rigid}}]
        \label{lem.vector bundle=fppf vector bundle}
    Let $X$ be a rigid-analytic variety and $F$ be a vector bundle on $X_\an$. Then we have the following identifications.
    \begin{enumerate}
        \item $ R\Gamma_\an(X,F)\xrightarrow{\simeq } R\Gamma_\et(X,F).$
        \item $ \Vect(X_\et) \xrightarrow{\simeq} \Vect(X_\an)$
        \item $\Vect((\Rig/X)_\tau)\xrightarrow{\simeq} \Vect(X_\tau) $ for $\tau\in \{\an,\et \}$.
    \end{enumerate}
\end{lemma}

Due to this lemma, we do not distinguish vector bundles in these topologies in the following and denote them as $\Vect(-)$.

\begin{lemma}
\label{lem.vector bundles statisfies abstract blowup squares}
    Let $p_Z:\Bl_Z(Y)\to Y$ be a smooth blowup of a nowhere dense smooth rigid-analytic variety $Y$. Suppose $E$ is the exceptional divisor and $F\in \Vect(Y)$ is a vector bundle on $Y$. Then we have the following:
    \begin{enumerate}
    \item There is a fibre sequence
    $$R\Gamma(Y,F ) \to R\Gamma(\Bl_Z(Y),F) \oplus R\Gamma(Z,F)\to R\Gamma(E,F).$$
    \item The pullback gives an equivalence
    $$\Vect({Y})\simeq \Vect({Z})\times_{\Vect({E})}\Vect{({\Bl_Z (Y)})}$$
    of groupoids.
    \end{enumerate}    
\end{lemma}

\begin{proof}
    Denote closed immersions $Z\to Y$, $E\to \Bl_Z(Y)$ as $i$ and $i'$, respectively, and the projection $E\to Z$ as $p_Z'$. Let $\mathcal{I}$ as the ideal sheaf of $E$. Since the statements are local we may assume that $Y=\Spa(A)$ and $Z=\Spa(A/I)$ are both affinoid. In this case every vector bundle on $Y$ is of the form $\widetilde{M}$ for some flat $A$-module $M$ by \cite[Theorem B.5.2]{zavyalov2024quotients}. Hence we may identify 
    $$\Vect(Y_a)\xrightarrow{\simeq }\Vect(Y)$$
    and similarly for $Z$. Moreover, since $Y$ is affinoid and blowups of $Y$ are relative analytifications of the algebraic blowups, by relative GAGA, c.f. \cite[Theorem 9.4.1]{fujiwara2013foundations}, we have 
    $$\Vect(\Bl_{Z_a}(Y_a))\xrightarrow{\simeq }\Vect(\Bl_Z(Y)),\ 
     \Vect(E_a)\xrightarrow{\simeq }\Vect(E)  $$
    where $E_a$ is the exceptional divisor associated to the algebraic blowup. In the sequential of the proof, we will not distinguish between the algebraic vector bundles and the analytic vector bundles.

    We first prove (1). Suppose $F$ is of the form $\widetilde{M}$ for some flat $A$-module $M$. Since $M$ is the direct summand of some free $A$-module, we may assume further that $F$ is trivial. This is proved by \cite[Proposition 5.1.1]{guo2019hodge}. In particular, since $Y$ is affinoid, statement (1) implies that there is a short exact sequence,
    $$0\to \Gamma(Y,\mathcal{O}_{Y})\to \Gamma(\Bl_{Z}(Y),\mathcal{O}_{\Bl_{Z}(Y)}) \oplus \Gamma(Z,\mathcal{O}_{Z}) \to \Gamma(E,\mathcal{O}_{E})\to 0.$$

    We now prove (2). For any $F_1,F_2 \in \Vect(Y_\et)$, notice that $$\Hom(F_1,F_2)=\Hom(\mathcal{{O}}_Y,\underline{\Hom}(F_1,F_2))= \Gamma(Y,\underline{\Hom}(F_1,F_2)).$$
    Hence one can deduce the full faithfulness of the pullback functor immediately from the above short exact sequence. As for essential surjectivity, we have to show that for any vector bundles $F_1\in \Vect(Z)$ and $F_2\in \Vect(\Bl_Z(Y))$ with an isomorphism $ p_Z'^*F_1\simeq i'^*F_2$, there is a vector bundle $G\in \Vect(Y)$ gluing this datum. By Zariski descent and faithfully flat descent, we may reduce to the case where $A$ is $I$-adically complete. Hence by the deformation theorem and the fact that $Y$ is affinoid, we know that $\Vect(Z)\simeq \Vect(Y)$. Therefore, there is a vector bundle $G\in \Vect(Y)$ corresponding to $F_1\in \Vect(Z)$ and an isomorphism $\psi_0:i'^*p_Z^*G \simeq i'^*F_2$. Now it reduces to show that there is a lift $\psi:p_Z^*G \simeq F_2$ of $\psi_0$ on $\Bl_Z(Y)$. By Grothendieck's existence theorem, c.f. \cite[Theorem 08BE]{stacks-project}, it suffices to show that $\psi_0$ lifts to any nilpotent thickening $E_n$ where $E_n$ is the closed subscheme defined by the ideal sheaf $\mathcal{I}^{n+1}$. By induction, we may assume there is a lift $\psi_n$ on $E_n$. Viewing $\psi_0$ as a section in $i'^*\underline{\Hom}(p_Z^*G, F_2)$, the obstruction class that lift from $E_n$ to $E_{n+1}$ lies in 
    $$H^1(\Bl_Z(Y),\underline{\Hom}(p_Z^*G, F_2)\otimes \mathcal{I}^n/\mathcal{I}^{n+1} )=H^1(E,p_Z'^*i^*\underline{\mathrm{End}}(G)\otimes \mathcal{I}^n/\mathcal{I}^{n+1} )$$
    where the isomorphism is given by the $\mathcal{O}_E$-module structure on $\mathcal{I}^n/\mathcal{I}^{n+1}$ and projection formula. By \cite[II Theorem 8.24]{Hartshorne} and \cite[Lemma 30.8.4]{stacks-project} we have
    $$H^j(E, \mathcal{I}^n/\mathcal{I}^{n+1} )=0$$
    for all $n\ge 1$ and $j\ge 1$. We will show that
    $$H^j(E,p_Z'^*N \otimes\mathcal{I}^n/\mathcal{I}^{n+1} )=0$$
    for $j\ge 1$ and any coherent sheaf $N$ on $Z$ which shows the vanishing of the obstruction class. To this end, we first write $N$ as a quotient of $\mathcal{O}_Z^k$ for some $k$ and denote the kernel as $K$. Since $p_Z'$ is flat and $\mathcal{I}^n/\mathcal{I}^{n+1}$ is a line bundle on $E$, we have a short exact sequence
    $$0\to p_Z'^*K\otimes\mathcal{I}^n/\mathcal{I}^{n+1} \to \mathcal{O}_E^k \otimes\mathcal{I}^n/\mathcal{I}^{n+1} \to p_Z'^*N\otimes\mathcal{I}^n/\mathcal{I}^{n+1}  \to 0. $$
    Taking the associated long exact sequence, we have 
    $$ \cdots\to  H^j(E,\mathcal{O}_E^k\otimes\mathcal{I}^n/\mathcal{I}^{n+1} ) \to H^j(E,p_Z'^*N\otimes\mathcal{I}^n/\mathcal{I}^{n+1}) \to H^{j+1}(E,p_Z'^*K\otimes\mathcal{I}^n/\mathcal{I}^{n+1})\to \cdots .$$
     The first term vanishes by the claim. By dimension bounds and induction, we know that the third term vanishes. Hence we have the desired vanishing of the second term. 
\end{proof}

Finally we present the proof of Theorem \ref{thm.descent of vector bundles}. This is a standard proof, c.f. \cite[Theorem 3.6]{geisser2006arithmetic}.

\begin{proof}[Proof of Theorem \ref{thm.descent of vector bundles}]
    We first verify 
    $$R\rho_{X,*}\rho_{X}^* F=F$$
    for all smooth $X$ and vector bundle $F$. This implies that $\rho_X^*$ is fully faithful. Since the statement is local, we may assume that $X$ is quasicompact. We want to show that the cone
    $$G\coloneqq \mathrm{cofib}(F\to R\rho_{X,*}\rho_X^*F)$$
    vanishes. It suffices to show that $H^i_\et(Y,G)$ vanishes for all $i\in \mathbb{Z}$ and $Y\in \Sm/X$. Assume the opposite that there is a nonzero element $x\in H^i_\et(Y,G)$ and $Y$ is a smooth rigid-analytic variety with the minimal dimension having such a property. Since $\rho_Y^* \simeq \rho_Y^*R\rho_{Y,*}\rho_Y^*$, we get $\rho_Y^*G\simeq 0$, i.e. there is a $\eh$-cover of $Y$ such that $G$ is equivalent to $0$ when pulled back to this cover. Since $G$ is already a sheaf in the \'etale topology, the cohomology class $x$ can not be trivialized by an \'etale cover. Hence by \cref{prop. guo decompose}, it suffices to show that $x$ does not vanish when restricted to a smooth blowup $p_Z:\Bl_Z(Y)\to Y$ along a nowhere dense closed subset $Z$. Denote the exceptional divisor as $E$. By Lemma \ref{lem.vector bundles statisfies abstract blowup squares} (1), there is a commutative diagram with exact rows
    \[
\begin{tikzcd}
    H^i_\et(Y,F)\arrow[d,"\tau_Y"]\arrow[r]& H^i_\et(Z,F)\oplus H^i_\et(\Bl_Z(Y),F)\arrow[d,"\tau_Z"',"\tau_{\Bl_Z(Y)}"] \ar[r] & H^i_\et(E,F)  \ar[d,"\tau_E"]\\ 
    H^i_\eh(Y,\rho_Y^*F)\arrow[r]& H^i_\eh(Z,\rho_X^*F)\oplus H^i_\eh(\Bl_Z(Y),\rho_Y^*F) \arrow[r]& H^i_\eh(E,\rho_Y^*F).
\end{tikzcd}
\]
By the assumption of the minimality of the dimension of $X$, $\tau_Z$ and $\tau_E$ are isomorphisms. Therefore $p_Z^*$ is injective on $H^i_\et(-,F)$ showing that $p_Z^*(x)\not= 0$ whence a contradiction.

Now we prove that $\rho_X^*$ is essentially surjective for all smooth rigid-analytic varieties $Y$. Clearly, the statement is true for dimension 0. Hence by induction on dimension, we may assume that the statement is true for smooth rigid-analytic varieties $Z$ with $\dim(Z)< \dim(Y)$.
Suppose $F$ is a vector bundle on $Y_\eh$. Again, by \cref{prop. guo decompose}, it suffices to show that if $F$ is trivial after pulling back to a smooth blowup along a smooth nowhere dense closed subset $Z$, then $F$ is already locally free. By induction hypothesis, $\Vect(Z_\eh)$ is equivalent to $ \Vect(Z)$. Hence we get a object in $\Vect({Z})\times_{\Vect({E})}\Vect{({\Bl_Z (Y)})}$. By Lemma \ref{lem.vector bundles statisfies abstract blowup squares} (2), we get a vector bundle $F'$ on $Y$. Thus $F$ is the essential image of $F'$ since $\rho_Y^*F'$ and $F$ agree on the $\eh$-cover $\Bl_Z(Y)\coprod Z$ of $Y$.
\end{proof}

\subsection{Descent of differentials}
Let $\pi:\Rig_\eh\to \Rig_\et$ be the projection of sites and for any $X\in \Rig$, let $\pi_X:X_\eh\to (\Rig/X)_\et$ be its localization at $X$. By slightly abusing notations, we also denote the projection to the small \'etale site of $X$ as $\pi_X$. Note that there is also a natural projection $\sigma_X: X_v \to X_\eh$.

\begin{definition}
\label{def. absolute differentials on et and h}
    Denote $\Omega^j_\et$ as the sheaf of absolute differentials on $\Rig_\et$ given by 
    $$\Omega^j_\et:\Rig_\et \to \Ab$$
    sending $Y$ to $\Gamma(Y,\Omega^j_Y)$. When localized at some rigid-analytic variety $Y$, we also denote it as $\Omega^j_{\et}$. Similarly we denote $\Omega_{\eh}^j$ as the $\eh$-sheafification of the sheaf of absolute differentials on $\Rig_\et$, namely
    $$\Omega^j_\eh\coloneqq\pi^*\Omega_\et^j.$$
\end{definition}

\begin{lemma}
    Let $X$ be a rigid-analytic variety. Then $\Omega^j_\eh$ on $X_\eh$ it is equivalent to the sheaf $R^j\sigma_{X,*}\widehat{\mathcal{O}}_{X_v}(j)$. In particular it unfolds the sheaf sending $Y\in \Sm/X$ to $\Gamma (Y,\Omega^j_Y)$.
\end{lemma}
\begin{proof}
    This is by \cite[Proposition 2.25]{heuer2022line}.
\end{proof}

Like the Deligne-Dubois complex, $\eh$-differentials and the \'etale sheaf of differentials coincide on smooth spaces.

\begin{proposition}[{\cite[Theorem 4.0.2]{guo2019hodge}}]
\label{cor. differential coincide on sm}
    Let $X$ be a smooth rigid-analytic variety. Then for any $i\in \mathbb{N}$, we have $R\pi_{X,*}\Omega^i_\eh=\Omega^i_X$.
\end{proposition}

\begin{warning}
\label{warning. big differential and relative differential}
    We warn the reader that for some rigid-analytic variety $Y$, there are two different notions of differentials on the big \'etale site $(\Rig/Y)_\et$. The first one is the absolute differentials $\Omega^j_{\et}$ defined in \cref{def. absolute differentials on et and h}. The other one is the pullback of the differential sheaf $\Omega^j_Y$ on the small \'etale site $Y_\et$. They are different sheaves since they are different when restricting to some rigid-analytic variety over $Y$ which is not \'etale. Moreover, even when $Y$ is smooth, $\Omega^j_\et$ is not a vector bundle on $(\Rig/Y)_\et$. By Lemma \ref{lem.vector bundle=fppf vector bundle}, all the vector bundles on $(\Rig/Y)_\et$ are pulled back from the vector bundles on $Y_\et$. However, $\Omega^j_\et$ is not pulled back from $Y_\et$ as the rank may vary when restricted on different objects in $(\Rig/Y)_\et$ which are not \'etale over $Y$. For the same reason $\Omega^j_\et$ is not even a coherent sheaf on $(\Rig/Y)_\et$. It is only a sheaf of $\mathcal{O}_{\et}$-module. However, it is clear that its restriction to $Y_\et$ is still $\Omega^j_Y$, namely $\Id_{Y,*}\Omega^j_\et=\Omega^j_Y$. Consequently, there is a natural map $\Id_{Y}^*\Omega^j_Y\to \Omega^j_\et$ on $(\Rig/Y)_\et$. Applying the same argument to $\eh$-topology, we conclude that $\Omega^j_\eh$ is not a coherent $\mathcal{O}_\eh$-module.
\end{warning}

\subsection{Descent of Higgs bundles}
In this subsection, we assemble everything together and prove that on a smooth rigid-analytic variety $X$, Higgs bundles on $X_\eh$ are equivalent to Higgs bundles on $X_\et$.

\begin{definition}
    A Higgs bundle $(E,\theta)$ on $X_\tau$ where $\tau\in \{ \et,\eh\}$, is a finite locally free $\mathcal{O}_{X_\tau}$-module $E$ equipped with a $\mathcal{O}_{X_\tau}$-linear morphism 
    $$\theta:E\to E\otimes\Omega^1_{X_\tau}(-1)$$
    such that $\theta\wedge\theta=0$. Denote the groupoid of Higgs bundles on $X_\tau$ as $\Higgs(X_\tau)$. Given a Higgs field $\theta$ on $X_\tau$, it is equivalent to giving a global section on $\Gamma(X,\underline{\mathrm{End
    }}(E)\otimes \Omega^1_{X_\tau}(-1))$ such that its image in $\Gamma(X,\underline{\mathrm{End
    }}(E)\otimes \Omega^2_{X_\tau}(-2))$ vanishes.
\end{definition}

\begin{proposition}\label{prop. compare eh higgs with et higgs}
    For any smooth rigid-analytic variety $X$, there is an equivalence of groupoids
    $$\pi_{X,*}:\mathrm{Higgs}(X_\eh) \xrightarrow{\simeq }\mathrm{Higgs}(X_\et).$$
\end{proposition}
\begin{proof}
    Given a Higgs bundle $(E,\theta)$ on $X_\eh$, by \cref{thm.descent of vector bundles}, $E$ is pulled back from a vector bundle on the \'etale site. Hence by \cref{cor. differential coincide on sm} and the projection formula of the ringed topoi, the functor $\pi_{X,*}$ defines a $\mathcal{O}_X$-linear morphism $\pi_{X,*}E\to \pi_{X,*}E\otimes \Omega^1_X(-1)$. Again by projection formula, we know that $\pi_{X,*}\theta\wedge\pi_{X,*}\theta=0$ showing that $\pi_{X,*}\theta$ is indeed a Higgs field. We then construct the converse of $\pi_{X,*}$. Given an \'etale Higgs bundle $(E,\theta)$ on $X$, one can pull back to the big \'etale site via $\Id_X^*$. After composing with the natural map in Warning \ref{warning. big differential and relative differential}, we get a morphism of $\mathcal{O}_\et$-modules
    $$ \Id_X^*E \to \Id_X^*E\otimes \Id_X^*\Omega^1_X(-1)\to \Id_X^*E\otimes\Omega^1_\et(-1).$$
    After further applying $\pi_X^*$ we get a morphism of $\mathcal{O}_h$-modules. Since pullback is exact and strongly symmetric monoidal, it clearly preserves the Higgs field condition. Denote the functor as ${\pi_X^*}'$. We need to verify that $\pi_{X,*}{\pi_X^*}' \simeq \Id$ and vice versa.

    To show that $\pi_{X,*}{\pi_X^*}' \simeq \Id$, note first that by \cref{thm.descent of vector bundles}, \cref{cor. differential coincide on sm} and the projection formula, we have $\pi_{X,*}{\pi_X^*}' E \simeq E$ and $\pi_{X,*}{\pi_X^*}' (E\otimes \Omega^1_X(-1)) \simeq E\otimes \Omega^1_X(-1)$ for any \'etale Higgs bundle $(E,\theta)$. Hence it suffices to show that $\pi_{X,*}{\pi_X^*}'  \theta \simeq \theta$. To this end, observe that 
    $$ \Gamma(X,\underline{\mathrm{End}}(E)\otimes \Omega^1_X(-1)) \xrightarrow{\simeq } \Gamma(X,\Id_X^*\underline{\mathrm{End}}(E)\otimes \Omega^1_\et(-1) )\xrightarrow{\simeq } \Gamma(X,\pi_X^*\underline{\mathrm{End}}(E)\otimes \Omega^1_h(-1) )$$
    are all isomorphic to each other. The first isomorphism is by \cite[Lemma 03YU]{stacks-project} and the fact that $\Id_{X,*}\Omega^1_\et(-1)\simeq \Omega^1_X(-1)$. The second isomorphism is clear since $\underline{\mathrm{End}}(E)$ is a vector bundle and by the same reasoning above. The isomorphisms above actually show the converse direction as well.
\end{proof}

\begin{corollary}
\label{cor.Higgs bundle has h descent for smooth rigid-analytic varieties}
    The presheaf 
    $$\Higgs: \Sm_\eh \to \mathrm{Groupoids}$$
    sending $Y$ to $\Higgs(Y)$ is an $\eh$-sheaf.
\end{corollary}

We will then prove that Higgs bundles on $\eh$-site have bounded cohomological dimension. To this end, we first specify the meaning of the cohomology of a Higgs bundle in $\eh$-topology.

\begin{definition}
\label{def.definition of dolbeault cohomology}
    Let $X$ be a rigid-analytic variety and $(E,\theta)$ be a Higgs bundle on $X_\tau$, where $\tau\in \{ \et,\eh\}$. Consider the coconnective complex 
    $$C_\Higgs(E,\theta)\coloneq( E\to E\otimes \Omega^1_{X_\tau}(-1)\to E\otimes \Omega^2_{X_\tau}(-2)\to \cdots)$$
    We define the Dolbeault cohomology (or the cohomology of the Higgs bundle) of $(E,\theta)$ as the hypercohomology of $C_\Higgs(E,\theta)$, namely 
    $$R\Gamma_\tau(X,(E,\theta))\coloneq R\Gamma_\tau(X,C_\Higgs(E,\theta)).$$
\end{definition}
\begin{remark}
    We explain one of the main differences between $\tau=\et$ and $\tau=\eh$ in the definition even when $X$ is smooth. Suppose $X$ has dimension $n$. Then on $X_\et$, the Higgs complex is a bounded complex in cohomological degree $[0,n]$. On the other hand, since $\Omega^j_\eh \neq 0$ for $j\ge n+1$ even when $X$ is smooth, the Higgs complex on $X_\eh$ is an unbounded coconnective complex. However, we will prove later that Doldeault cohomology always has the correct cohomological dimension bound. Moreover, for any Higgs bundle $(E,\theta)$ on the $\eh$-topology of a smooth rigid-analytic variety $X$, we have 
    $$R\Gamma_\eh(X,(E,\theta))=R\Gamma_\et(X,\pi_{X,*}(E,\theta)).$$
\end{remark}

\begin{lemma}
\label{lem.derived pushforward of Higgs complex}
    Let $X$ be a smooth rigid-analytic variety and $(E,\theta)$ be a Higgs bundle on $X_\eh$. Then we have $R\pi_{X,*}C_\Higgs(E,\theta)$ lives in cohomological dimension $[0,2n]$ in $D^b_\mathrm{coh}(X_\et)$.
\end{lemma}
\begin{proof}
    The spectral sequence for hypercohomology says that $$E_1^{p,q}=R^p\pi_{X,*}C^q_\Higgs(E,\theta) \Rightarrow R^{p+q}\pi_{X,*}C_\Higgs(E,\theta).$$
    Since $C^q_\Higgs(E,\theta)=E\otimes \Omega^q_\eh(-q)$ by \cite[Corollary 6.0.4]{guo2019hodge}, $E_1^{p,q}=0$ unless $0\le p,q\le n$. Hence $R\pi_{X,*}C_\Higgs(E,\theta)$ lies in cohomological dimension $[0,2n]$ as desired. Note that by the proof of \cite[Proposition 6.0.1]{guo2019hodge}, all the higher direct images of $\Omega^i_\eh$ are coherent. The assertion for coherent sheaves follows immediately from the fact that $E_1^{p,q}$ are all coherent and therefore all terms in $R\pi_{X,*}C_\Higgs(E,\theta)$. 
\end{proof}
\begin{corollary}
\label{cor.rough bound of h-cohomology}
    Let $X$ be a smooth rigid-analytic variety and $(E,\theta)$ be a Higgs bundle on $X_\eh$. Then the Dolbeault cohommology $R\Gamma_h(X,(E,\theta))$ of $(E,\theta)$ lives in cohomological dimension $[0,3n]$.
\end{corollary}
\begin{proof}
    By Leray spectral sequence we have $$E^{p,q}_2=H^p_\et(X,R^q\pi_{X,*}C_\Higgs(E,\theta))\Rightarrow H^{p+q}_\eh(X,(E,\theta)).$$
    On the other hand, by \cref{cor.rough bound of h-cohomology}, $R^q\pi_{X,*}C_\Higgs(E,\theta)$ is a coherent sheaf and therefore by rigid-analytic GAGA, $E^{p,q}_2=0$ unless $0\le p\le n$. Combine with the previous Lemma we win.
\end{proof}

\begin{remark}
    Of course the bound in \cref{cor.rough bound of h-cohomology} is not sharp. Later in the paper, once we establish the comparison between pro-\'etale cohomology and Dolbeault cohomology, we will know that it has the correct cohomological dimension bound.
\end{remark}

Using hypercovers, one immediately deduces the following.

\begin{corollary}
\label{cor. rough dolbeault cohomological dimension bounds}
    Let $X$ be a rigid-analytic variety and $(E,\theta)$ be a Higgs bundle on $X_\eh$. Then we have $R\pi_{X,*}C_\Higgs(E,\theta)$ lives in cohomological dimension $[0,2n]$ in $D^b_\mathrm{coh}(X_\et)$. In particular, the Dolbeault cohommology $R\Gamma_\eh(X,(E,\theta))$ of $(E,\theta)$ lives in cohomological dimension $[0,3n]$
\end{corollary}
\begin{proof}
    Take a smooth $\eh$-hypercover $X_\bullet$ of $X$. Then the coherence directly follows from the spectral sequence in \cref{prop.simplicial space spectral sequence}. As for the cohomological dimension bounds, one uses the blowup excision as in \cref{prop.h-sheaf=fppf sheaf + abstract blowups} and proceeds the same proof as in \cref{prop.cohomological dimension bounds of proet VB}. Since we will get the correct cohomological dimension bounds there, we do not spell out the proof here.
\end{proof}

\newpage
\section{Simplicial rigid-analytic varieties and $p$-adic Simpson correspondence}
To generalize the $p$-adic Simpson correspondence from smooth rigid-analytic varieties to singular ones, we need to take a smooth hypercover of the latter. We spell out some technical details of simplicial rigid-analytic varieties in this section.

\subsection{Generalities on simplicial rigid-analytic varieties}
We recall and build some basic notions on simplicial rigid-analytic varieties in this section that will be used later.

\begin{definition}
    A simplicial rigid-analytic variety $X_\bullet$ is a functor $X_\bullet:\Delta^{\rm op}\to \Rig$ where $\Delta$ is the category of linearly ordered finite sets. In practice, we image $X_\bullet$ as a sequence of rigid-analytic varieties 
    $$\cdots   \substack{\longrightarrow\\[-1em] \longleftarrow \\[-1em] \longrightarrow\\[-1em] \longleftarrow \\[-1em] \longrightarrow} X_1 \substack{\longrightarrow\\[-1em] \longleftarrow \\[-1em] \longrightarrow}   X_0 $$
    with face maps $d^j_n: X_n\to X_{n-1}$ and degeneracy maps $s^j_n: X_n\to X_{n+1}$ for $0\le j\le n$ satisfies certain relations given by $\Delta$, c.f. \cite[016B]{stacks-project}. A morphism between simplicial rigid-analytic varieties is a natural transformation. We will denote the category of simplicial rigid-analytic varieties as $\Rig_{\Delta}$.

    Moreover, we say a morphism of simplicial rigid-analytic variety $X_\bullet\to Y_\bullet$ satisfies property \textbf{P} if $X_n\to Y_n$ satisfies property \textbf{P} for each $n$ where \textbf{P}=\'etale, smooth, flat, proper, etc.
\end{definition}

\begin{remark}
    In the rest of this section, one can freely replace $\Delta$ by its variants: the augmented finite ordered sets $\Delta^+$, the $m$-truncated finite ordered sets $\Delta_{\le m}$ and the augmented $m$-truncated finite ordered sets $\Delta^+_{\le m}$. Hence one can get different kinds of simplicial spaces. It is not very hard to connect different variants by slice, skeleton and coskeleton functor. These operations are well-known, so we refer the reader to \cite[09VI]{stacks-project} and \cite{conrad2003cohomological} for more details. For practice, we will mainly use the augmented truncated version. And when talking about an $n$-truncated (augmented) simplicial object, we always mean its $n$-coskeleton. 
\end{remark}


\begin{definition}\label{def. simplicial sheaves in various topologies}
 For a simplicial rigid-analytic variety $X_\bullet$ and $\tau\in \{\an, \et, \proet,v,\eh\} $, we define $X_{\bullet,\tau}$ as the simplicial $\tau$-site of $X_\bullet$ (when $\tau=v$, replace $X_\bullet$ by $X^\diamond_\bullet$) as follows:
 \begin{enumerate}
     \item An object of $X_{\bullet,\tau}$ is $U_n\to X_n\in X_{n,\tau}$ for some $n$.
     \item A morphism from $(U_n\to X_n)$ to $(U_m\to X_m)$ is given by a commutative diagram
     \[
     \begin{tikzcd}
         U_n \arrow[r,"f"] \arrow[d] & U_m \arrow[d]\\
         X_n \arrow[r, "X_\bullet(\phi)"] & X_m 
     \end{tikzcd}
     \]
     where $f$ is a morphism in suitable category and $\phi:[m]\to [n]$ is morphism in $\Delta$.
     \item A cover of $(U_{n}\to X_n)$ is just cover of $U_n$ in $X_{n,\tau}$.
 \end{enumerate}
We define a $\tau$-sheaf on $X_\bullet$ to be a sheaf on $X_{\bullet,\tau}$. We denote the $\tau$-sheaves of abelain groups as $\Ab(X_{\bullet,\tau})$. We drop $\tau$ if there is no confusion. As in the nonsimplicial case, we denote $\nu:X_{\bullet,\proet}\to X_{\bullet,\et}$, $\lambda:X_{\bullet,\proet}\to X_{\bullet,\an}$ as the projections.

\end{definition}

\begin{remark}
    Note that it is not true in general that for any \'etale cover $\{U\to X_n\}$, we have an \'etale morphism of simplicial rigid-analytic variety $U_\bullet\to X_\bullet$ whose restriction to degree $n$ is $U$. Nevertheless, one can always dominate $U\to X_n$ by the restriction to degree $n$ of some \'etale surjective map $U_\bullet\to X_\bullet$ using the coskeleton functor, c.f. \cite[Proposition 1.5]{friedlander1982etale}.
\end{remark}

Sheaves on a simplicial site is "cosimplicial" in the following sense.

\begin{lemma}[{\cite[09VM]{stacks-project}}]
\label{lem.explicitsimplicialsheaf}
    Let $X_\bullet$ be a simplicial rigid-analytic variety. Then $F^\bullet$ is a sheaf of abelian groups on $X_{\bullet,\tau}$ if and only if it consists of the following datum:
    \begin{enumerate}
        \item A system of sheaves $F^n$ on $X_{n,\tau}$.
        \item For any $\phi:[m]\to [n]$ in $\Delta$, there is a morphism $[\phi]:X_\bullet(\phi)^{-1}(F^m)\to F^n$.
        \item For any $\phi:[m]\to [n]$ and $\psi:[k]\to [m]$, we have $[\phi]\circ X_\bullet(\phi)^{-1}[\psi]=[\phi\circ\psi]$.
    \end{enumerate}
    We say $F^\bullet$ is cartesian if $[\phi]:X_\bullet(\phi)^{-1}(F^m)\to F^n$ is an isomorphism for all $\phi:[m]\to [n]$ in $\Delta$. 
\end{lemma}

\begin{remark}\label{rem.restriction is exact}
    It is not hard to see that the restriction functor $\Ab(X_\bullet)\to \Ab(X_n)$ is exact for every $n$, c.f. \cite[Lemma 09WG]{stacks-project}.
\end{remark}

We will use the following spectral sequence multiple times.

\begin{proposition}[{\cite[09WJ]{stacks-project}}]
\label{prop.simplicial space spectral sequence}
    There is a functorial spectral sequence in $K^\bullet\in D^+(X_{\bullet,\tau})$:
    $$E_1^{p,q}=H^q_\tau(X_p,K^p )\Rightarrow H_\tau^{p+q}(X_\bullet,K^\bullet)$$
    where we view $K^\bullet$ as a sequence of complexes by Lemma \ref{lem.explicitsimplicialsheaf}. 
\end{proposition}

\begin{remark}
\label{rem.nonabelian simplicial space spectral sequence}
    We also need a nonabelian analogue of the proposition above. Following the proof of {\cite[09WJ]{stacks-project}}, one can prove a nonabelian analogue in the following sense: let $K$ be a sheaf of $n$-truncated spaces on $X_{\bullet,\tau}$, then there is a functorial homotopy spectral sequence
    $$\pi_{-q}(K(X_p)) \Rightarrow \pi_{-p-q}(K(X_\bullet)).$$
\end{remark}

\begin{definition}
\label{def.simplicialpcohsheaf}
    There is a natural sheaf of rings $\mathcal{O}_{X_\bullet,\tau}$ on $X_{\bullet,\tau}$ which restriction on each $X_n$ is just $\mathcal{O}_{X_n,\tau}$. It gives a ringed topos $(X_{\bullet,\tau},\mathcal{O}_{X_\bullet,\tau})$. Then we define sheaf of $\mathcal{O}_{X_\bullet,\tau}$-module and (pseudo)coherent $\mathcal{O}_{X_\bullet,\tau}$-module by general nonsense.

    A sheaf of $\mathcal{O}_{X_\bullet,\tau}$-module is cartesian if $[\phi]:X_\bullet(\phi)^{*}(F^m)\to F^n$ is an isomorphism for all $\phi:[m]\to [n]$ in $\Delta$. 
\end{definition}

\begin{lemma}[{\cite[0D7M]{stacks-project}}]
\label{lem.pseudocoherent=pseudocoherent on each degree+cartesian}
    Let $F^\bullet$ be a sheaf of $\mathcal{O}_{X_\bullet}$-module. Then $F^\bullet$ is pseudocoherent if and only if it is cartesian and pseudocoherent for each $F^n$.
\end{lemma}

\begin{remark}
\label{rem. dual in simplicial O-modules}
    Let $X_\bullet$ be a simplicial rigid-analytic variety and $\tau\in \{\an,\et \}$. One needs to be careful when taking duals in the category of $\mathcal{O}_{X_\bullet,\tau}$-modules. Let $F^\bullet$ be a $\mathcal{O}_{X_\bullet,\tau}$-module, then the dual of $F^\bullet$ is defined by 
    $$\underline{\Hom}(F^\bullet,\mathcal{O}_{X_\bullet,\tau})(U_n)=\underline{\Hom}(F^\bullet|_{U_n},\mathcal{O}_{X_\bullet,\tau}|_{U_n})$$
    for some $U_n\in X_{n,\tau}$. However, the latter is not necessarily $\underline{\Hom}(F^n|_{U_n},\mathcal{O}_{U_n,\tau})$ since the restriction should be interpreted by localizing the big site $X_{\bullet,\tau}$ at the object $U_n$ instead of restricting to the small site $U_{n,\tau}$. 
\end{remark}

\begin{remark}
    When $\tau\in \{\proet,v \}$, we usually use the complete structure sheaf $\widehat{\mathcal{{O}}}_{X_\bullet,\tau}$ as defined in \cite{scholze2013adic}. In this spirit, we denote $\nu^*$ as the following functor
\[\nu^{\ast}:=\hat{\mathcal{O}}_{X_{\bullet, \proet}}\otimes_{\nu^{-1}{\mathcal{O}}_{X_{\bullet, \et}}}\nu^{-1}(-):\mathrm{Mod}(X_{\bullet, \et},\mathcal{O}_{X_{\bullet,\et}})\to \mathrm{Mod}(X_{\bullet, \proet},\hat{\mathcal{O}}_{X_{\bullet,\proet}}).\]
Similarly to the non-simplicial case, one can also define (proétale) period sheaves ($\mathbb{A}^\bullet_\inf$, $\bdr$, etc) as in \cite{scholze2013adic}.
\end{remark}

\begin{example}
\label{eg.simplicial cotangent sheaf is not VB or dualizable}
    For any smooth simplicial rigid-analytic variety $X_\bullet$. We define the simplicial differential $\Omega^1_{X_\bullet}$ of $X_{\bullet}$ to be the sheaf of $\mathcal{O}_{X_\bullet,\et}$-module on $X_\bullet$ such that the $n$-th term on the $n$-th rigid-analytic variety $X_n$ is $\Omega^1_{X_n}$. As opposed to the non-simplicial case, $\Omega^1_{X_\bullet}$ is no longer a coherent sheaf anymore since it is not cartesian and thus does not satisfy Lemma \ref{lem.pseudocoherent=pseudocoherent on each degree+cartesian}.
\end{example}




\subsection{Simplicial Higgs bundles}
In the rest of this subsection, we prove some properties of Higgs bundles in this setting that will be used later. One needs to be slightly careful in the simplicial setting since $\Omega^1_{X_{\bullet}}(-1)$ is no longer coherent by Example \ref{eg.simplicial cotangent sheaf is not VB or dualizable}. Moreover due to the subtlety of taking duals as remarked in \cref{rem. dual in simplicial O-modules}, we do not use the linear dual of $\Omega^1_{X_{\bullet}}(-1)$ in the following. 

\begin{definition}\label{def.coherent B}
Let $X_\bullet$ be a simplicial rigid-analytic variety and $\tau\in \{\et,\eh,v \}$. A Higgs field on $X_{\bullet,\tau}$ is a vector bundle $E^\bullet$ on $X_{\bullet,\tau}$ with a morphism 
$$\theta: E^{\bullet}\to E^{\bullet}\otimes \Omega^1_{X_{\bullet,\tau}}(-1)$$
such that $\theta \wedge \theta =0$ and is $\mathcal{O}_{X_{\bullet,\tau}}$-linear (where in the case of $v$-site, $\Omega^1_{X_{\bullet,\tau}}(-1)$ simply means the pullback of $\Omega^1_{X_{\bullet,\et}}(-1)$). Again the groupoid of Higgs bundles on $X_{\bullet,\tau}$ is denoted as $\Higgs
(X_{\bullet,\tau})$. Similarly as in \cref{def.definition of dolbeault cohomology}, we define the Dolbeault cohomology of a Higgs bundle $(E^\bullet,\theta)$ on $X_{\bullet,\tau}$ as the hypercohomology of 
$$C_\Higgs(E^\bullet,\theta)\coloneq(E^\bullet\to E^\bullet\otimes \Omega^1_{X_{\bullet,\tau}}(-1)\to E^\bullet\otimes \Omega^2_{X_{\bullet,\tau}}(-2)\to \cdots).$$
For a similar reason as before, $C_\Higgs(E^\bullet,\theta)$ is only a connective complex even when $X_\bullet$ is smooth.

\end{definition}

\begin{remark}
    In the following, we also require that Higgs bundles are "cartesian" in the sense that the only failure to be cartesian is given by the sheaf of differentials. Specifically speaking, we require that for any $\phi:X_m\to X_n$, the Higgs field $\theta_m$ on $X_m$ is given by the composition of $\phi^*\theta_n$ and $\phi^*\Omega^1_{X_n}(-1) \to \Omega^1_{X_m}(-1)$, namely
    $$\theta_m:\phi^*E^n \to \phi^*E^n \otimes  \phi^* \Omega^1_{X_{n}}(-1) \to \phi^*E^n \otimes\Omega^1_{X_{m}}(-1).$$
\end{remark}

\begin{remark}
    We warn the reader here that one should not expect a generalization of the Hodge-Tate log exact sequence as in \cite[Theorem 2.4]{heuer2023p} since there is no good analogue of spectral variety in this context. Specifically speaking, let $X_\bullet$ be a smooth rigid-analytic variety and $(E^\bullet,\theta)$ be a Higgs bundle on $X_{\bullet,\et}$. Then it is not clear whether there is a simplicial rigid-analytic variety $S_\bullet $ that maps to $X_\bullet$ such that $S_n\to X_n$ corresponds to the spectral variety of $(E^n,\theta_n)$ for every $n$. Indeed, let $f:X\to Y$ be a map of smooth rigid-analytic varieties and $(E,\theta)$ be a Higgs field on $Y_\et$ and $f^*(E,\theta)$ be its pullback. Suppose $S_Y$ and $S_X$ be the corresponding spectral varieties, then we have the following diagram 
    \[
    \begin{tikzcd}
        S_X \ar[d,hookrightarrow] & S_Y/X \arrow[dr, phantom, "\ulcorner", very near start]\ar[r]\ar[l]\ar[d,hookrightarrow] & S_Y \ar[d,hookrightarrow]\\
        T^*X(-1) \ar[dr]& T^*Y(-1)/X \arrow[dr, phantom, "\ulcorner", very near start]\ar[d] \ar[l] \ar[r] & T^*Y(-1) \ar[d]\\
         & X \ar[r] & Y .
    \end{tikzcd}
    \]
    However, the arrow from $S_Y/X \to S_X$ is in the opposite direction to make a simplicial spectral variety.
\end{remark}

Finally we prove an analogous statement of \cref{prop. compare eh higgs with et higgs} for simplicial rigid-analytic varieties. We still denote the projection $X_{\bullet,\eh}\to X_{\bullet,\et}$ as $\pi_{X_\bullet,*}$.

\begin{proposition}
\label{prop.simplicial et Higgs and eh Higgs}
    Let $X_\bullet$ be a smooth simplicial rigid-analytic variety. Then the natural projection gives an equivalence of groupoids
    $$\pi_{X_\bullet,*}: \Higgs(X_{\bullet,h})\to \Higgs(X_{\bullet,\et}).$$
\end{proposition}
\begin{proof}
    One construct the inverse functor ${\pi_{X_\bullet}^*}'$ similarly. And to check that ${\pi_{X_\bullet}^*}'$ is indeed an inverse, it suffices to check degreewise by \cref{prop.simplicial space spectral sequence} which follows from \cref{prop. compare eh higgs with et higgs}.
\end{proof}

Geometric Sen operator also makes sense in this context.

\begin{theorem}[c.f. {\cite[Theorem 4.8]{heuer2023p}}]\label{thm. canonical proetale Higgs field on simplicial}
    Let $X_\bullet$ be a smooth simplicial rigid-analytic variety. Then for any vector bundle $\mathcal{E}^\bullet$ on $X_{\bullet,v}^\diamond$, there exists a unique Higgs field 
    $$\eta_{\mathcal{E}^{\bullet}}: \mathcal{E}^{\bullet}\to \mathcal{E}^{\bullet}\otimes \Omega^1_{X_{\bullet}}(-1)$$
    on $X_{\bullet, v}^\diamond$ such that the following conditions hold:
    \begin{enumerate}
        \item The association $\mathcal{E}^\bullet \mapsto \eta_{\mathcal{E}^{\bullet}}$ is functorial in $X_\bullet$ and $\mathcal{E}^{\bullet}$;
        \item If $X_\bullet$ is constant, then $\eta_{\mathcal{E}^{\bullet}}$ is the Higgs field defined in \cite[Theorem 4.8]{heuer2023p};
        \item $\eta_{\mathcal{E}^{\bullet}}=0$ if and only if $\mathcal{E}^{\bullet}$ comes from a vector bundle on $X_{\bullet, \et}$.
    \end{enumerate}
\end{theorem}
\begin{proof}
    First we notice that for constant $X_\bullet$, \cite[Theorem 4.8]{heuer2023p} provides a way of associating Higgs field $\eta_{\mathcal{E}^{\bullet}}$ to a pro-\'etale vector bundle $\mathcal{E}^\bullet$ on $X_{\bullet,v}^\diamond$. To extend such a construction to all of smooth simplicial rigid-analytic varieties in a functorial way, we observe that for a general smooth simplicial rigid-analytic variety $X_\bullet$ and a vector bundle $\mathcal{E}^\bullet$ on $X_{\bullet,v}^\diamond$, when restricted to each $n$, the above construction produces a Higgs field $\eta_{\mathcal{E}^{n}}$ on the vector bundle $\mathcal{E}^n$ over $X_{n,v}^\diamond$. Moreover, the desired compatibility of face maps and degeneracy maps follows from the functoriality in $X$ stated in \cite[Theorem 4.8]{heuer2023p}. Consequently, we obtain a Higgs field $\eta_{\mathcal{E}^{\bullet}}$ on $\mathcal{E}^\bullet$. As for (3), everything can be checked locally, hence degreewise, which then follows from the usual statement in geometric Sen theory.
\end{proof}

\subsection{$p$-adic Simpson correspondence for smooth proper simplicial rigid-analytic varieties}

Putting everything together, we now come to the desired theorem of $p$-adic Simpson correspondence for smooth proper simplicial rigid-analytic varieties.

\begin{theorem}[c.f. {\cite[Theorem 5.1]{heuer2023p}}]
\label{thm.simplicial p-adic simpson correspondence}
     Let $X_\bullet $ be a smooth proper simplicial rigid-analytic variety with a lift $\mathbb{X}_\bullet$ to ${B^+_\dR}/\xi^2$. Fix an exponential $\mathrm{exp}:C\to 1+\mathfrak{m}$. Then there is an equivalence of symmetric monoidal categories:
         \begin{eqnarray*}
			S_{\mathbb{X}_\bullet,\exp}:\Vect(X_{\bullet,\proet}) \xrightarrow{\simeq }\Higgs(X_{\bullet,\et}) \xrightarrow{ \simeq}\Higgs(X_{\bullet,\eh})
		\end{eqnarray*}
    Moreover, given a morphism $f:X_\bullet\to Y_\bullet$ of smooth proper simplicial rigid-analytic varieties with a compatible morphism between their lifts $\tilde{f}:\mathbb{X}_\bullet\to \mathbb{Y}_\bullet$. Then there is a natural transformation 
        \[\begin{tikzcd}
        \Vect(Y_{\bullet,\proet})\arrow[r,"S_{\mathbb{Y}_\bullet,\exp}"]  \arrow[d,"f^*"] & \Higgs(Y_{\bullet,\et}) \arrow[d,"f^*"]  \\
        \Vect(X_{\bullet,\proet})\arrow[r,"S_{\mathbb{X}_\bullet,\exp}"]   \arrow[ur,Rightarrow]          & \Higgs(X_{\bullet,\et})           
		\end{tikzcd}\]
        induced by $\tilde{f}$ that is compatible by the compositions.
\end{theorem}
\begin{proof}
    For each fixed $n$, the choice of the lift $\mathbb{X}_\bullet$ and the exponential $\mathrm{exp}$ induces an equivalence 
    \begin{equation*}
        S_{\mathbb{X}_n,\exp}:\Vect(X_{n,\proet}) \xrightarrow{\simeq }\Higgs(X_{n,\et}) \xrightarrow{ \simeq}\Higgs(X_{n,\eh}).
    \end{equation*}
   Indeed, the first equivalence (and the construction of the functor $S_{\mathbb{X}_n,\exp}$) is \cite[Theorem 5.1]{heuer2023p}, while the second equivalence follows from \cref{prop.simplicial et Higgs and eh Higgs}. Moreover, when $n$ varies, $S_{\mathbb{X}_n,\exp}$ are compatible with each other by the functoriality given in \cite[Theorem 5.1 (2)]{heuer2023p}. Thus we can define $S_{\mathbb{X}_\bullet,\exp}$ degreewise. Finally by \cref{rem.nonabelian simplicial space spectral sequence}, we get the desired equivalence by comparing the $E_1$-page of the spectral sequence.

\end{proof}

We finally discuss the cohomological comparison. 

\begin{proposition}[c.f. {\cite[Theorem 5.5]{heuer2023p}}]
\label{prop.simplicial cohomological comparison}
    Let $X_\bullet$ be a smooth proper rigid-analytic variety with a lift $\mathbb{X}_\bullet$ to $B^+_\dR/\xi^2$. Fix an exponential $\exp:C\to 1+\mathfrak{m}$. By \cref{thm.simplicial p-adic simpson correspondence}, for any pro-\'etale vector bundle $\mathcal{E}^\bullet $, there is an \'etale Higgs bundle $(E^\bullet ,\theta)$ given by $S_{\mathbb{X}_\bullet,\exp}(\mathcal{E})$. Then there is a natural isomorphism 
    $R\nu_*\mathcal{E}^\bullet \simeq C_\Higgs(E^\bullet ,\theta)$ inducing 
    $$ R\Gamma_\proet(X_\bullet,\mathcal{E}^\bullet) \simeq R\Gamma_\et(X_\bullet,(E^\bullet,\theta)).$$
\end{proposition}
\begin{proof}
    By \cref{prop.simplicial space spectral sequence}, it suffices to construct a map 
    $$ C_\Higgs(E^\bullet,\theta)\to R\nu_*\mathcal{E}^\bullet$$
    such that on each $X_n$ it is compatible with the one defined in \cite[Theorem 5.5]{heuer2023p}. However this can be easily defined as presheaves since for some $U_n \in X_{n,\et}$ we have 
    $$C_\Higgs(E^\bullet,\theta)(U_n)=C_\Higgs(E^n,\theta_n)(U_n)\to R\nu_*\mathcal{E}^\bullet=(R\nu_*\mathcal{E}^n)(U_n)$$
    and is given by the map from \cite[Theorem 5.5]{heuer2023p} and the functoriality from \cite[Theorem 5.1]{heuer2023p}.
\end{proof}

\newpage
\section{$p$-adic Simpson correspondence for proper rigid-analytic varieties}
In this section, we use all the ingredients above to prove a $p$-adic Simpson correspondence for any proper rigid-analytic variety over $C$, removing the smoothness condition in the work of \cite{heuer2023p}.

\subsection{Cohomological dimension bounds}

We need to prove that the cohomological dimension bounds of pro-\'etale vector bundles on a rigid-analytic variety are the same as expected, i.e., twice as the dimension of the rigid-analytic variety. To this end, we need the following result, which is well-known to experts.
\begin{lemma}
\label{lem.proetale VB has h-descent}
    The presheaf
    $$ \Rig_\eh\to \mathrm{Groupoids}$$
    sending $Y$ to $\Vect(Y_\proet)$ is a $\eh$-sheaf. 
\end{lemma}
\begin{proof}
By \cite[Theorem 3.5.8]{kedlaya2016relative2}, for any affinoid perfectoid space $X$, the category of pro-\'etale vector bundles on $X$ is equivalent to the category of $v$-vector bundles on $X$. Let $Y$ be any rigid-analytic variety. Since the affinoid perfectoid objects form a basis for the pro-\'etale site of $Y$ by \cite[Proposition 4.8]{scholze2013adic}, one immediately deduce that the category of pro-\'etale vector bundles on $Y$ is equivalent to the category of $v$-vector bundles on $Y^\diamond$, namely
$$\Vect(Y_\proet)\simeq \Vect(Y^\diamond_v).$$
Hence the category of pro-\'etale vector bundles satisfies $v$-hyperdescent. By \cref{rem.eh vs v topology}, $v$-topology is finer than $\eh$-topology, the consequence easily follows.
\end{proof}

By (the proof of) the previous lemma, we do not distinguish pro-\'etale vector bundles and $v$-vector bundles in the following. Now we get to prove the cohomological bound. The following argument was suggested to us by Bogdan Zavyalov.

\begin{proposition}
\label{prop.cohomological dimension bounds of proet VB}
    Let $X$ be a rigid-analytic variety of dimension $n$ and $\mathcal{E}$ be a pro-\'etale vector bundle on $X$. Then $R\nu_*\mathcal{E}$ is concerntrated in cohomological dimension $[0,n]$ of $D^b_\mathrm{coh}(X_\et)$. In particular, the cohomology group $R\Gamma_\proet(X,\mathcal{E})$ lies in cohomological degree $[0,2n]$.
\end{proposition}
\begin{proof}
We first prove the case where $X$ is smooth. We may assume that $X=\Spa(A)$ is affinoid. Moreover, by \cite[Theorem 1.3.4]{zavyalov2021almost} and further replacing $X$ by a finite \'etale cover if necessary, we may assume there is locally free $\mathcal{O}^+_{X^\diamond_v}$-sheave $\mathcal{E}^+$ such that $\mathcal{E}^+[1/p]=\mathcal{E}$. Choose an admissible formal model $\mathfrak{X}=\Spf(R)$ of $X$ and denote the natural morphism by $t: X_\et \to \mathfrak{X}_\et$ and the composition by $\Psi: X^\diamond_v \to \mathfrak{X}_\et$. By \cite[Lemma 7.8.7]{zavyalov2021almost}, we may choose a formal model $\mathfrak{X}$ such that it admits a finite open cover $\{\mathfrak{U}_i\}_{i\in I}$ where the restriction of $\mathcal{E}^+$ to each $\mathfrak{U}_i$ is small, i.e. can be trivialized by a finite \'etale cover. We first claim that 
$$R\Psi_* \mathcal{E}^{+}\in \widehat{D}^+_{\mathrm{qcoh}}(\mathfrak{X}_\et),\ (R\Psi_* \mathcal{E}^{+})^a\in \widehat{D}^{[0,n]}_\mathrm{acoh}(\mathfrak{X}_\et)^a.$$
For the first claim, it suffices to show that the natural morphism 
$$\widetilde{H^i(X^\diamond_v,\mathcal{E}^+)}\to R^i\Psi_*\mathcal{E}^+$$
is an isomorphism. To this end, it reduces to show that for any affine open $\mathfrak{U}=\Spf(R[1/f]^\wedge)$ with its rigid-generic fibre $U$, the natural morphism of base change
$$ H^i_v(X^\diamond,\mathcal{E}^+)\widehat{\otimes}_{R}R[1/f]^\wedge \to H^i_{v}(U^\diamond,\mathcal{E}^+)$$
is an isomorphism. Choose a strictly totally disconnected perfectoid cover $X_\infty=\Spa(A_\infty,A_\infty^+)$ of $X$, then by \cite[Corollary 7.5.2]{zavyalov2021almost} the pullback $U_\infty=X_\infty\times_X U$ is also a strictly totally disconnected perfectoid space. Since strictly totally disconnected perfectoid spaces are acyclic objects for $\mathcal{E}^+$, both cohomology groups are computed by the Čech cohomologys of $\Gamma((X_\infty^{\bullet/X})^\diamond,\mathcal{E}^+)$ and $\Gamma((U_\infty^{\bullet/U})^\diamond,\mathcal{E}^+)$ respectively. Since the base change is $p$-completely flat, it reduces to showing that the natural morphism 
$$ \Gamma((X_\infty^{j/X})^\diamond,\mathcal{E}^+)\widehat{\otimes}_{R}R[1/f]^\wedge \to \Gamma(({U}_\infty^{j/U})^\diamond,\mathcal{E}^+)$$
is an isomorphism. Since strictly totally disconnected perfectoid space does not have any higher cohomology for $\mathcal{E}^+$, we have $\Gamma(({U}_\infty^{j/U})^\diamond,\mathcal{E}^+/p)=\Gamma(({U}_\infty^{j/U})^\diamond,\mathcal{E}^+)/p$. Hence by derived Nakayama lemma, c.f. \cite[Lemma 0G1U]{stacks-project}, it suffices to check the above isomorphism for $\mathcal{E}^+/p$ where it is proved by \cite{zavyalov2021almost}. Hence we win. As for the second claim for almost coherence, since we have already proved that $R\Psi_*\mathcal{E}^+$ is quasicoherent, by \cite[Proposition 7.1.6]{bhattpadicnote} and the assumption that $\mathfrak{X}$ is affine, it reduces to showing that the $R\Psi_*(\mathcal{E}^+/p)$ is almost coherent. This is proved by \cite[Theorem 7.8.8]{zavyalov2021almost}. As for the cohomological dimension bounds, it follows from the almost analogue of the derived Nakayama lemma and \cite[Theorem 7.10.3]{zavyalov2021almost}.

To prove that $R\nu_*\mathcal{E}$ is in fact coherent, it suffices to show that $t^*R\Psi_*\mathcal{E}^+=R\nu_*\mathcal{E}$. For any rational open set $U$, by taking stricly totally disconnected perfectoid cover of $U$, it is clear that $$\Gamma(U^\diamond,t^*R^i\Psi_*\mathcal{E}^+)=\Gamma(U^\diamond,R^i\nu_*\mathcal{E}).$$
Thus it suffices to show that for any (standard) \'etale morphism $V\to U$, the fibre of $$\Gamma(U,R^i\nu_*\mathcal{E}^+)\widehat{\otimes}\Gamma(V,\mathcal{O}^+)\to \Gamma(V,R^i\nu_*\mathcal{E}^+)$$
is killed by bounded $p$-torsion. To this end, since the restriction of $\mathcal{E}^+$ to each $\mathfrak{U}_i$ is small and hence its rigid-generic fibre. Therefore, we may reduce to the case where $\mathcal{E}^+$ is trivial. Then we win by \cite[Lemma 6.18]{scholze2013adic}.

Finally, we extend the statement to the general case by using blowup excisions. Again, we may assume that $X$ is quasicompact. Since pro-\'etale vector bundles satisfy $v$-descent and hence $\eh$-descent by \cref{lem.proetale VB has h-descent}, we may freely use descents of $\eh$-sheaves in the following. Passing to its reduction and use that $\eh$-coverings contain universal homeomorphisms, we may also assume that $X$ is reduced. Hence by resolution of singularities, there is a nowhere dense Zariski closed subset $Z$ of $X$ whose blowup $Y$ is smooth with exceptional divisor $E$. By blowup excisions of $\eh$-topology, i.e. \cref{prop.h-sheaf=fppf sheaf + abstract blowups}, there is an fibre sequence
$$R\Gamma_\proet(X,\mathcal{E})\to R\Gamma_\proet(Z,\mathcal{E})\oplus R\Gamma_\proet(Y,\mathcal{E})\to R\Gamma_\proet(E,\mathcal{E})$$
and similarly for $R\nu_*\mathcal{E}$. Hence by induction on dimension and the smooth case, we get the desired statements.

\end{proof}
\begin{remark}
    We remark here that there are several other ways to prove the above results. For example, since this problem is local, one may use Faltings' small correspondence in the smooth case, c.f. \cite[Section 3]{faltings2005p} to use the corresponding cohomology of small Higgs bundles to compute the pro-\'etale cohomology. Another way is to use the newly developed theory of perfectoid crystals as in \cite{bhattpadicnote}. Namely, one can realize pro-\'etale vector bundles (\'etale locally) as rationalized almost overconvergent \'etale $\hat{\mathcal{O}}^+$-vector bundles. One then uses \cite[Theorem 5.3.5]{bhattpadicnote} to relate it with perfectoid crystals and uses the affineness of perfectoidization in this case to deduce the desired cohomological bounds. In this way one can avoid using the resolution of singularities.
\end{remark}

\subsection{$p$-adic Simpson correspondence for proper rigid-analytic varieties}
\begin{theorem}
\label{thm. p-adic simpson for proper rigid-analytic variety}
    Let $X$ be a proper rigid-analytic variety. Fix a 1-truncated smooth proper $\eh$-hypercover of $X$ with a lift $\mathbb{X}_\bullet$ to $B^+_\dR/\xi^2$ and an exponential map $\exp: C\to 1+\mathfrak{m}$. Then there is an equivalence of symmetric monoidal categories:
    $$S_{\mathbb{X}_\bullet,\exp}:\Vect(X_\proet) \xrightarrow{\simeq} \Higgs(X_\eh).$$
    Moreover, given a morphism $f:X\to Y$ of proper rigid-analytic varieties with 1-truncated smooth proper $\eh$-hypercover $X_\bullet,Y_\bullet$ and compatible morphism between their lifts $\tilde{f}:\mathbb{X}_\bullet\to \mathbb{Y}_\bullet$. Then there is a natural transformation 
        \[\begin{tikzcd}
        \Vect(Y_{\proet})\arrow[r,"S_{\mathbb{Y}_\bullet,\exp}"]  \arrow[d,"f^*"] & \Higgs(Y_\eh) \arrow[d,"f^*"]  \\
        \Vect(X_{\proet})\arrow[r,"S_{\mathbb{X}_\bullet,\exp}"]   \arrow[ur,Rightarrow]          & \Higgs(X_\eh)    
		\end{tikzcd}\]
        induced by $\tilde{f}$ that is compatible by the compositions.

\end{theorem}

We will show this equivalence by choosing a liftable smooth proper hypercover $X_\bullet$ of $X$ and use the simplicial $p$-adic Simpson correspondence to descend on $X$. To this end, recall the following results of \cite{guo2019hodge}.

\begin{proposition}[{\cite[Proposition 7.4.4]{guo2019hodge}}]
\label{prop. proper rigid-analytic variety has liftable hypercovers}
    Let $X$ be a proper rigid-analytic variety. Then for any natural number $n$, there exists a $n$-truncated augmented simplicial rigid-analytic variety $\mathbb{X}_\bullet \to \mathbb{X}$ which is flat tft over $B^+_\dR/\xi^2$. Moreover, the reduction to $C$ induces an $n$-truncated smooth proper $\eh$-hypercover of $X$.
\end{proposition}

In order to discuss functoriality, one needs to take hypercovers with lifts in a compatible way. The following corollary guarantees the existence of such hypercovers.

\begin{corollary}
\label{cor. weak functoriality of the hypercovers}
    Let $X\to Y$ be a morphism of proper rigid-analytic varieties. Then for any natural number $n$, there are $n$-truncated augmented simplicial rigid-analytic varieties $\mathbb{X}_\bullet \to \mathbb{X}$ and $\mathbb{Y}_\bullet \to \mathbb{Y}$ that are flat tft over $B^+_\dR/\xi^2$ inducing smooth proper $\eh$-hypercovers reduction to $C$ and a commutative diagram
    \[\begin{tikzcd}
			\mathbb{X}_\bullet \arrow[r]  \arrow[d] & \mathbb{Y}_\bullet \arrow[d]  \\
			\mathbb{X} \arrow[r]             & \mathbb{Y}  .            
		\end{tikzcd}\]
\end{corollary}
\begin{proof}
    In order to give the proof, we briefly recall the spreading out argument in \cite[Corollary 13.16]{BhattMorrowScholzeIntegralPadicHodge}. To start with, we denote $K$ as $W[1/p]$ where $W$ is the Cohen ring of the residue field $k$ of $C$ and fix an inclusion $K\to C$. For a finite diagram of proper rigid-analytic varieties $\{T_i\}_{i\in I}$, suppose we can spread out $\{T_i\}_{i\in I}$ to a finite diagram $\{\tilde{T}_i \}$ proper and flat over a smooth rigid-analytic variety $S/K$ such that $\{T_i\}_{i\in I}$ arises as the fibre at some $C$-point of $S$. Shrink $S$ finitely many times if necessary, one may assume further that each $\tilde{T}_i $ is smooth over $S$. Since $S$ is smooth, the nilpotent thickening $ \Spa C \to \Spa B^+_\dR/\xi^2$ over $K$ can be lifted to $S$. Thus, one can pullback the diagram $\{\tilde{T}_i \}$ along the lift $\Spa B^+_\dR/\xi^2 \to S$. To find such an $S$, we need to find a finite diagram of proper flat formal models $\{ \mathfrak{T}_i\}_{i\in I}$ over $\mathcal{O}_C$ such that $\{T_i\}_{i\in I}$ arises as the rigid-generic fibre of them. Given such a diagram of formal models, we then consider the deformation functor of the special fibres:
    $$\mathrm{Def}_{\{ \mathfrak{T}_{k,i}\}_{i\in I}}:\mathrm{Art} \to \Set$$
    where $\mathrm{Art}$ is the category of artinian $W$-algebras with residue field $k$. Standard argument, c.f \cite[0E3S]{stacks-project}, then shows that the tangent space of $\mathrm{Def}_{\{ \mathfrak{T}_{k,i}\}_{i\in I}}$ is a finite dimensional $k$-vector space. Thus \cite[Lemma 06IW]{stacks-project} shows such deformation functor admits a versal object, namely there exists a complete noetherien local $W$-algebra $R$ and a diagram of proper flat formal schemes $\{ \mathfrak{T}_{R,i}\}_{i\in I}$ over $R$ deforming $\{ \mathfrak{T}_{k,i}\}_{i\in I}$ such that the classifying map is formally smooth. Consider the ind-category of $\mathrm{Art}$ as in the proof of \cite[Proposition 13.15]{BhattMorrowScholzeIntegralPadicHodge} produces a map $R\to \mathcal{O}_C$. Take the rigid generic fibre of $\{ \mathfrak{T}_{R,i}\}_{i\in I}$ over $\Spf (R)$ gives a diagram of rigid-analytic varieties over $S\coloneq \Spa(R[1/p],R)$. Replace $S$ with a locally closed subspace if necessary, we get the desired diagram.

    Now we explain how to find a diagram of proper flat simplicial formal schemes
    \[\begin{tikzcd}
			\mathfrak{X}_\bullet \arrow[r]  \arrow[d] & \mathfrak{Y}_\bullet \arrow[d]  \\
			\mathfrak{X} \arrow[r]             & \mathfrak{Y}             
		\end{tikzcd}\]
    over $\mathcal{O}_C$ such that its rigid generic fibre forms smooth proper $\eh$-hypercovers of $X$ and $Y$ respectively. We first take an $n$-truncated smooth proper $\eh$-hypercover $Y_\bullet\to Y$ which admits a proper flat formal model $\mathfrak{Y}_\bullet\to \mathfrak{Y}$ by \cite[Claim 7.4.5]{guo2019hodge}. Base change along some proper flat formal model $\mathfrak{{X}}\to \mathfrak{Y}$ of $X$, we get a augmented simplicial formal scheme $\mathfrak{X}\times_{\mathfrak{Y}}\mathfrak{Y_\bullet} \to \mathfrak{X}$. Pick a smooth proper $\eh$-cover $X_0$ of $X\times_Y Y_0$ such that it admits a formal model $\mathfrak{X}_0$ that maps to $\mathfrak{X}\times_\mathfrak{Y}\mathfrak{Y}_0$. Assume by induction that $\mathfrak{X}_\bullet$ is already constructed up to degree $m$. Consider the canonical map 
    $$\gamma:(\mathrm{cosk}_m\mathrm{sk}_m X_\bullet)_{m+1} \to (\mathrm{cosk}_m\mathrm{sk}_m Y_\bullet)_{m+1}.$$ 
    By definition the natural map $Y_{m+1 }\to (\mathrm{cosk}_m\mathrm{sk}_m Y_\bullet)_{m+1}$ is a smooth proper $\eh$-cover. Pull back $Y_{m+1 }$ along $\gamma$ gives a proper $\eh$-cover $X'_{m+1}$ of $(\mathrm{cosk}_m\mathrm{sk}_m X_\bullet)_{m+1}$. Take a smooth proper $\eh$-cover $N$ of $X'_{m+1}$ along with a proper flat formal model $\mathfrak{N}$ compatible with the resolution. Finally, we use \cite[Theorem 4.12]{conrad2003cohomological} to glue such $N$ and $\mathfrak{N}$ to get the rigid-analytic variety and formal model in degree $m+1$ respectively. 
\end{proof}



We now prove the main theorem.

\begin{proof}[Proof of \cref{thm. p-adic simpson for proper rigid-analytic variety}]
    The existence of a 1-truncated smooth proper $\eh$-hypercover $X_\bullet$ of $X$ which admits a compatible lift $\mathbb{X}_\bullet \to \mathbb{X}$ to $B^+_\dR/\xi^2$ is guaranteed by \cref{prop. proper rigid-analytic variety has liftable hypercovers}. Since both $\Vect(X_\proet)$ and $\Higgs(X_\eh)$ are 1-truncated $\eh$-sheaf, they are (homotopy) equalizers of $\mathrm{sk}_1X_\bullet$, namely
    $$\Vect(X_\proet)=\mathrm{eq}(\Vect(X_{0,\proet})\substack{\longrightarrow\\[-1em] \longrightarrow} \Vect(X_{1,\proet})),\ \Higgs(X_\eh)=\mathrm{eq}(\Higgs(X_{0,\eh})\substack{\longrightarrow\\[-1em] \longrightarrow} \Higgs(X_{1,\eh})).$$
    By \cref{prop.simplicial et Higgs and eh Higgs}, we have $\Higgs(\mathrm{sk_1}X_{\bullet,\eh})\simeq \Higgs(\mathrm{sk_1}X_{\bullet,\et})$. By \cref{thm.simplicial p-adic simpson correspondence}, there is an equivalence of symmetric monoidal categories 
    $$\Vect(\mathrm{sk}_1X_{\bullet,\proet}) \to \Higgs(\mathrm{sk}_1 X_{\bullet,\et}).$$
    Taking the limit on both sides gives the desired equivalence. The functoriality follows immediately from the one of simplicial smooth rigid-analytic varieties.
\end{proof}

Finally we compare the cohomology of a pro-\'etale vector bundle $\mathcal{E}$ with the cohomology of its corresponding Higgs field on a rigid-analytic variety $X$.
\begin{theorem}\label{thm. cohomology comparison}
     Let $X$ be a proper rigid-analytic variety of dimension $n$. Take a $3n+2$-truncated smooth proper $\eh$-hypercover of $X$ with a lift $\mathbb{X}_\bullet$ to $B^+_\dR/\xi^2$. Given a pro-\'etale vector bundle $\mathcal{E}$ on $X$, by \cref{thm. p-adic simpson for proper rigid-analytic variety} it corresponds to a Higgs bundle $(E,\theta)=S_{\mathbb{X}_\bullet,\exp}(\mathcal{E})$ on $X_\eh$, then 
    \begin{equation*}
        R\Gamma_\proet(X, \mathcal{E})\simeq R\Gamma_\eh(X, (E,\theta)).
    \end{equation*}
    In particular, $R\Gamma_\eh(X, (E,\theta))$ is concentrated on cohomological degree $[0,2n]$.
\end{theorem}
\begin{proof}
     By \cref{prop. proper rigid-analytic variety has liftable hypercovers}, there is a $3n+2$-truncated smooth proper $\eh$-hypercover $X_\bullet$ of $X$ which admits a compatible lift $\mathbb{X}_\bullet \to \mathbb{X}$ to $B^+_\dR/\xi^2$. Denote by $\rho: X_{\bullet}\to X$ the augmentation map. Then the cone of the natural map 
    \begin{equation*}
        R\Gamma_\proet(X, \mathcal{E}) \to R\Gamma_\proet(\mathrm{sk}_{3n+2}X_\bullet,\mathcal{E}^\bullet)
    \end{equation*}
    lives in cohomological degree $\geq 3n+1$ where $\mathcal{E}^\bullet=\rho^*\mathcal{E}$. Similarly, the cone of the natural map 
    \begin{equation*}
        R\Gamma_\eh(X, (E,\theta)) \to R\Gamma_\eh(\mathrm{sk}_{3n+2}X_{\bullet}, (E^{\bullet},\theta))
    \end{equation*}
    lives in cohomological degree $\geq 3n+1$. On the other hand, by \cref{prop.simplicial cohomological comparison} and \cref{prop.simplicial et Higgs and eh Higgs}, we have that 
    \begin{equation*}
        R\Gamma_\proet(\mathrm{sk}_{3n+2}X_{\bullet}, \mathcal{E}^{\bullet})\xrightarrow{\simeq} R\Gamma_\et(\mathrm{sk}_{3n+2}X_{\bullet}, (E^{\bullet},\theta))=R\Gamma_\eh(\mathrm{sk}_{3n+2}X_{\bullet}, (E^{\bullet},\theta)).
    \end{equation*}
    Combining all of the above arguments together, we deduce that 
    \begin{equation*}
        \tau^{\leq 3n}R\Gamma_\proet(X, \mathcal{E}) = \tau^{\leq 3n}R\Gamma_\proet(\mathrm{sk}_{3n+2}X_{\bullet}, \mathcal{E}^{\bullet}) \xrightarrow{\simeq} \tau^{\leq 3n}R\Gamma_\eh(\mathrm{sk}_{3n+2}X_{\bullet}, (E^{\bullet},\theta)) =
        \tau^{\leq 3n}R\Gamma_\eh(X, (E,\theta)).
    \end{equation*}
    By \cref{cor. rough dolbeault cohomological dimension bounds}, $R\Gamma_\eh(X, (E,\theta))$ is concentrated in cohomological degree $[0,3n]$. On the other hand, by \cref{prop.cohomological dimension bounds of proet VB}, $R\Gamma_\proet(X, \mathcal{E})$ is concentrated on cohomological degree $[0,2n]$, hence
    \begin{equation*}
        R\Gamma_\proet(X, \mathcal{E})= \tau^{\leq 3n}R\Gamma_\proet(X, \mathcal{E})\xrightarrow{\simeq} \tau^{\leq 3n} R\Gamma_\eh(X, (E,\theta)) = R\Gamma_\eh(X, (E,\theta))
    \end{equation*}
    as desired.
\end{proof}

Finally, we record a version of the geometric Sen operator in the singular cases that might be interested to some readers. By slightly abusing notations, we will identify $\Omega^1_X$ and $\Omega^1_\eh$ with their pullbacks to the $v$-site.

\begin{proposition}
    Let $X$ be a rigid-analytic variety. Then for any $v$-vector bundle $\mathcal{E}$ on $X_v^\diamond$, there exists a unique Higgs field
    $$\eta_\mathcal{E}: \mathcal{E} \to \mathcal{E}\otimes \Omega^1_\eh(-1)$$
    on $X_v^\diamond$ such that the following conditions hold:
    \begin{enumerate}
        \item The association $\mathcal{E}\to \eta_\mathcal{E}$ is functorial in $X$ and $\mathcal{E}$;
        \item If $X$ is smooth, then $\eta_\mathcal{E}$ is the composition of $\theta_\mathcal{E}:\mathcal{E}\to \mathcal{E}\otimes \Omega^1_X(-1)$ and $\Id\otimes \mathrm{can}: \mathcal{E}\otimes \Omega^1_X(-1) \to \mathcal{E}\otimes \Omega^1_\eh(-1)$ where $\theta_\mathcal{E}$ is the geometric Sen operator for smooth rigid-analytic varieties;
        \item We have $\eta_\mathcal{E}=0$ if and only if $\mathcal{E}$ comes from a vector bundle on $X_\eh$.
    \end{enumerate}
\end{proposition}
\begin{proof}
    It suffices to prove that such Higgs fields exist uniquely for quasicompact rigid-analytic varieties. To this end, let $X$ be a quasicompact rigid-analytic variety and $X_\bullet$ be a smooth $\eh$-hypercover of $X$ with augmented map $\rho:X_{\bullet,v}^\diamond \to X^\diamond_v$. By \cref{thm. canonical proetale Higgs field on simplicial}, there exists a unique Higgs field for $\rho^*\mathcal{E}$ given by 
    $$\eta_{X_\bullet}:\rho^*\mathcal{E} \to \rho^*\mathcal{E} \otimes \Omega^1_{X_\bullet}(-1) $$
    Applying $\rho_*$ and projection formula gives a Higgs field for $\mathcal{E}$ of the form 
    $$\rho_*\eta_{X_\bullet}:\mathcal{E} \to \mathcal{E}  \otimes \rho_* \Omega^1_{X_\bullet}(-1).$$
    Since $\Omega^1_{X_\bullet}(-1)$ can be identified with $\Omega^1_{X_\bullet,\eh}$ by \cref{prop.simplicial et Higgs and eh Higgs}, we have $\rho_* \Omega^1_{X_\bullet}(-1) \simeq \Omega^1_\eh(-1)$ on $X_v^\diamond$. This proves the existence of a Higgs field for $\mathcal{E}$. It remains to show that such a Higgs field is independent of the choice of $X_\bullet$. Suppose there is another smooth $\eh$-hypercover of $X$ given by $X'_\bullet$. By a similar argument as in \cref{cor. weak functoriality of the hypercovers}, we can find a smooth $\eh$-hypercover $X_\bullet''$ of $X$ which dominates both $X_\bullet$ and $X_\bullet'$, namely we have a commutative diagram
    \[
    \begin{tikzcd}
        X_\bullet''^\diamond \ar[r,"\alpha"] \ar[d,"\beta"] \ar[dr,"{\rho''}"] & X_\bullet'^\diamond \ar[d,"\rho'"]\\
        X_\bullet^\diamond \ar[r,"\rho"] & X^\diamond.
    \end{tikzcd}
    \]
    Since the Higgs field on $X^\diamond$, $X'^\diamond$ and $X''^\diamond$ are all canonical, we get $\eta_{X_\bullet}=\beta_*\eta_{X_\bullet''}$ and $\eta_{X_\bullet'}=\alpha_*\eta_{X_\bullet''}$. Hence we can identify both $\rho_*\eta_{X_\bullet}$ and $\rho'_*\eta_{X'_\bullet}$ with $\rho''_*\eta_{X_\bullet''}$ showing the desired uniqueness. The rest of the claim follows easily from the statement of \cref{thm. canonical proetale Higgs field on simplicial} and $\eh$-descent.
\end{proof}

\newpage
\bibliographystyle{amsalpha}
\bibliography{main,preprints}

\end{document}